\documentclass[12pt,oneside,reqno]{amsart}
\usepackage[T1]{fontenc}
\usepackage[frenchb,english]{babel}
\usepackage{indentfirst}
\usepackage{mathrsfs,cancel}  
\usepackage{bm,amsmath,amsthm}   
\usepackage{graphicx,enumerate,calc,lscape,color,mathdots}
\usepackage{amssymb}
\usepackage{diagbox}
\usepackage{fullpage}
\usepackage{amscd}
\usepackage{color}
\usepackage[alphabetic,lite]{amsrefs}
\usepackage{verbatim} 
\usepackage{icomma} 
\usepackage[mathscr]{euscript}

\newtheorem{Theorem}{Theorem}

\newtheorem{cor}[subsection]{Corollary}
\newtheorem{lemma}[subsection]{Lemma}

\theoremstyle{definition}  
\newtheorem{example}[subsection]{Example}

\newtheorem{remark}[subsection]{Remark}

\newcommand{ \calO}{{\mathcal O}}
\newcounter{theoremCounter}

\begin{document}

\title{On the 2-rank and 4-rank of the class group of some real pure quartic number fields}

\author{Mbarek Haynou}
\author{Mohammed Taous}

\address{Moulay Ismail University of Meknes. Faculty of Sciences and Technology\\
	P.O. Box 509-Boutalamine, 52 000 Errachidia}
\email{haynou\_mbarek@hotmail.com}
\email{taousm@hotmail.com}

\begin{abstract}
Let $K=\mathbb{Q}(\sqrt[4]{pd^{2}})$ be a real pure quartic number field and $k=\mathbb{Q}(\sqrt{p})$  its real quadratic subfield, where $p\equiv 5\pmod 8$ is a prime integer and  $d$ an odd square-free integer coprime to $p$. In this work, we calculate  $r_2(K)$, the  $2$-rank of the class group of $K$, in terms of the  number of prime divisors of $d$ that decompose or remain inert in $\mathbb{Q}(\sqrt{p})$, then we will deduce forms
of $d$ satisfying $r_2(K)=2$. In the last case,  the $4$-rank of the class group of $K$ is given too.
\end{abstract}
\maketitle
\section{\large \bf{Introduction and Notations}}
\label{introduction}
Let $K$ be a number field and  $Cl_2(K)$  its  $2$-class group, that is the $2$-Sylow subgroup of its class group $Cl(K)$.  We define  the $2$-rank and the $4$-rank of $Cl(K)$  respectively as follows: $r_2(K)=\dim_{\mathbb{F}_2}(Cl_2(K)/Cl_2(K)^2)$ and $r_4(K)=\dim_{\mathbb{F}_2}(Cl_2(K)^2/Cl_2(K)^4)$
where $\mathbb{F}_2$ is the finite field with $2$ elements.

Using  genus theory, many mathematicians  calculated $r_2(K)$ and $r_4(K)$ whenever  $K$ is a  number field having a subfield with odd class number. For instances, we mention the following works:
\begin{itemize}
	\item [$ \bullet $]	For biquadratic number fields, we indicate the works  \cite{AT08, AM01, AM04, ATZ18, Q09, T08}$\dots$.
	\item [$ \bullet $]	For cyclic quartic fields, we indicate the works \cite{Pa77, Pa78}$ \dots $.
	\item [$ \bullet $]	For pure quartic number fields, we indicate the works of Parry \cite{N5, P80}$ \dots $.
\end{itemize}
In the two last papers, Parry determined the exact power of $2$ dividing the class number of some pure quartic number fields.

In this paper, we consider the real pure quartic number fields  $K=\mathbb{Q}(\sqrt[4]{pd^{2}})=k(\sqrt{d\sqrt{p}})$, where $k=\mathbb{Q}(\sqrt{p})$  with $p \equiv 5 \pmod 8$ is a prime integer and $d$ is an odd square-free integer satisfying some conditions.

Our first goal is to determine the $2$-rank of the class group of $K$,  using the ambiguous class number formula in $K/k$ (\cite{C33}):
\begin{equation*}\label{CNF1}
\#\mathcal{A}m(K/k) = h(k) \cdot \frac{2^{t-1}}{ [E_k:E_k \cap N_{K/k}(K^\times)]},
\end{equation*}

where
 \begin{itemize}
 	 \item [$ \bullet $]$t$ is the number of finite and infinite
 	 primes ramified in $K/k$,
 	\item [$ \bullet $]  $E_k$ is the unit group of $k$,
 	\item [$ \bullet $] $E_k \cap N_{K/k}(K^\times)$ is the subgroup of units that are norms of elements of $K$,
 	\item [$ \bullet $] $h(k)$ is the class number of $k$.
 \end{itemize}
 As $h(k)$ is odd, then the above formula implies that
 \begin{equation}\label{ACNF2}
 r_2(K) = t - e - 1,
 \end{equation}
 with $e$ is an integer defined by $2^{e} = [E_k:E_k \cap N_{K/k}(K^\times)]$.
Our second goal is to calculate the $4$-rank of the class group of fields $K$ satisfying  $r_{2}(K)=2$. For this, we will use  a formula provided by  Y. Qin (\cite{Q09}). This formula states that  the $4$-rank of the class group of  a quadratic extension $ k(\sqrt{\delta})$, where the base  field $k$ is of odd class number, is given by:
\begin{equation}\label{4RF}
 r_{4}(k(\sqrt{\delta}))= t-1-\mathrm{rank(}R_{k(\sqrt{\delta})/k}),
\end{equation} where $t$ is the number of ramified
primes in $k(\sqrt{\delta})/k$ and $R_{k(\sqrt{\delta})/k}$ is the following matrix:
\begin{equation}\label{RMG}
 R_{k(\sqrt{\delta})/k}=\left(\begin{array}{llll}
 \left(\dfrac{a_{1};\,\delta}{\mathcal{P}_{1}}\right) &$\dots$\left(\dfrac{a_{n};\,\delta}{\mathcal{P}_{1}}\right) &$\dots$  \left(\dfrac{a_{n+r};\,\delta}{\mathcal{P}_{1}}\right)\\
 & \dots & \dots \\
 \left(\dfrac{a_{1};\,\delta}{\mathcal{P}_{t}}\right) &$\dots$\left(\dfrac{a_{n};\,\delta}{\mathcal{P}_{t}}\right) &$\dots$  \left(\dfrac{a_{n+r};\,\delta}{\mathcal{P}_{t}}\right)
 \end{array} \right).
\end{equation}

It is a matrix of type $t\times(n+r)$ with coefficients in $\mathbb{F}_2$, called the generalized R\'edei-matrix,
where $(\mathcal{P}_{i})_{1\leq i\leq t}$ are the primes (finite and infinite) of $k$ which ramify in $ k(\sqrt{\delta})$, $(a_{j})_{1\leq j\leq n+r}$ is a family of elements of $k$  defined  by Y. Qin in \cite[{\textsection}2,  p. 27]{Q09} and $\left(\dfrac{-;\,\delta}{\mathcal{P}_{i}}\right)$ is the Hilbert symbol on $k$. Note that this matrix is given by replacing the 1's by 0's and the -1's by 1's.
For more details concerning the generalized R\'edei-matrix,  we refer the reader to \cite{Q09}.

During this paper, we adopt the following notations:
\begin{itemize}
\item $k=\mathbb{Q}(\sqrt{p})$  with $p \equiv 5 \pmod 8$ is a prime integer.
\item For $z \in k$,  $z'$ denotes the conjugate of $z$  over $\mathbb{Q}$.
\item $K=\mathbb{Q}(\sqrt[4]{pd^{2}})=k(\sqrt{d\sqrt{p}})$, where $d$ is an odd square-free integer such that $(p, d)=1$.
	\item $\mathcal{O}_k$:  the ring of integers of $k$.
	\item $r_{2}(K)$:  the $2$-rank of the class group $Cl(K)$.
	\item $r_{4}(K)$:  the $4$-rank of the class group $Cl(K)$.
	\item $\delta=d\sqrt{p}$.
	\item $\mathcal{P}_{i}$:  a prime ideal of $k$ which ramifies in $K$.
	\item $N_{K/k}(-)$:  the relative norm of $K$ to $k$.
	\item $R_{ K/k}$:  the generalized R\'edei-matrix of $K$.
	\item $\left(\frac{m}{n}\right)$:  the Legendre symbol.
	\item $\left[\frac{-}{\mathcal{P}_{i}}\right]$:  the quadratic symbol over $k$.
	\item $\left(\frac{ -, -}{\mathcal{P}_{i}}\right)$:  the Hilbert symbol over $k$.             			
	\item $\left(\frac{m}{n}\right)_{4}$:  the rational biquadratic symbol.
	\item  $\varepsilon_{p} $:  the fundamental unit of $k$.
	
\end{itemize}


Our main theorems are as follows. Their proofs will be given in Sections \ref{hay2} and \ref{hay1}. 
\begin{Theorem} \label{Tha} Suppose that $d = q_{1}\cdots q_{s}q'_{1}\cdots q'_{t}$ is an odd square-free integer such that  $q_{i}$ and $q'_j$ are distinct primes satisfying $\left(\frac{p}{q_i}\right)=-1$ and $\left(\frac{p}{q'_{j}}\right)=1$ for each $i$, $j$. Then the $2$-rank of $Cl(K)$ is
	$r_{2}(K)=2t+s.$
\end{Theorem}
\begin{Theorem} \label{Thb}
	If $d=q$ is an odd prime such that $\left(\frac{p}{q}\right)=1$. Then
	\begin{enumerate}[{\indent\rm(1)}]
		\item If $q\equiv 1 \pmod 8$, then
		$r_{4}(K)=\begin{cases} 0, & \text{if $\left(\frac{ p }{q}\right)_{4}=-1 $,} \\
		1, & \text{if $\left(\frac{ p }{q}\right)_{4}=-\left(\frac{ q }{p}\right)_{4}=1 $,} \\
		2 ,& \text{if $\left(\frac{ p }{q}\right)_{4}=\left(\frac{ q }{p}\right)_{4}=1$.}
		\end{cases}    $
		\item If $q\equiv 3 \pmod 8$, then	$  r_{4}(K)=0 \hspace{12em}  $
		\item If $q\equiv 5 \pmod 8$, then
		$r_{4}(K)=\begin{cases} 0, & \text{if  $\left(\frac{ p }{q}\right)_{4}=-1 $,} \\
		1, & \text{if $\left(\frac{ p }{q}\right)_{4}=1$.}
		\end{cases}   \hspace{4em}  $
		\item If $q\equiv 7 \pmod 8$, then
		$r_{4}(K)=\begin{cases} 0, & \text{if $\left(\frac{ q }{p}\right)_{4}=1 $,} \\
		1, & \text{if $\left(\frac{ q }{p}\right)_{4}=-1 $.}
		\end{cases} \hspace{4em}$
	\end{enumerate}			
		
	\end{Theorem}
    \begin{Theorem} \label{Thc}
    If $d=q_{1}q_{2}$ where $q_{1}, q_{2}$ are odd prime integers
	such that $\left(\frac{p }{q_{1}}\right)=\left(\frac{p }{q_{2}}\right)=-1$,
	then
	$$  \hspace{6.5em} r_{4}(K)=\begin{cases} 1, & \text{if $(q_{1}, q_{2})\equiv (3, 3) \pmod 4 $,} \\
	0, & \text{otherwise.}
	\end{cases} $$		
    \end{Theorem}
We will analyze the behaviour of this non-Galois extension, $K$, of degree 4 in order to compare, in an other paper in preparation, the results obtained here in this paper with those that can one have when $K$ is replaced by a biquadratic field. The determination of the $2$-rank and $4$-rank of the class group of $K$, can be also among the most important properties used to study the following problems:
\begin{enumerate}[\indent\rm1.]
 \item Characterize the structure of the $2$-group $G=Gal(K^{(2)}_{2}/K)$  of $K$ where $K^{(2)}_{2}$ is its second Hilbert $2$-class field.
  \item  Determine the Hilbert $2$-class field tower of pure quartic field $K$.
   \item Study the capitulation of the $2$-ideal classes of the pure quartic field $K$  in the intermediate sub-extensions of $K^{(1)}_{2}/K$ where $K^{(1)}_{2}$ is the first Hilbert $2$-class field of $K$.
  \end{enumerate}
\section{\large \bf{The $2$-rank of $Cl(K)$}\label{hay2}}
In this section, we assume that $d = q_{1}\cdots q_{s}q'_{1}\cdots q'_{t}$ where $q_{i}$ and $q'_j$ are distinct primes satisfying $\left(\frac{p}{q_{i}}\right)=-1$, $\left(\frac{p}{q'_{j}}\right)=1$ for each $i$, $j$.

The relative discriminant of $ K/k $ is $\Delta_{K/k}=(4d\sqrt{p})$ (see \cite{H92}), then the finite primes ramified in $ K/k $ are
$(\sqrt{p}), \tilde{q_{1}}, \cdots, \tilde{q_{s}}$, $ \pi_{1}, \cdots, \pi_{t}$, $\tilde{\pi_{1}}, \cdots, \tilde{\pi_{t}}$ and $2_{I}$, where $pO_{k}=(\sqrt{p})^{2}$, $2_{I}=2O_{k}$, $ \pi_{i}\tilde{\pi}_{i}=q'_{i}O_{k}$ and $\tilde{q_{i}}=q_{i}O_{k} $ for each $ i $.
Regarding   infinite primes ramifying in $K/k$,  $k$ has two infinite primes $p_{\infty}$ and $\tilde{p}_{\infty}$ which correspond respectively to the two following $\mathbb{Q}$-embeddings:  $i_{p_{\infty}}: \sqrt{p}\longmapsto -\sqrt{p}$ and $i_{\tilde{p}_{\infty}}: \sqrt{p}\longmapsto\sqrt{p}$. As $i_{p_{\infty}}$ can be extend to the two $\mathbb{Q}$-embeddings:  $j_{p_{\infty}}: \sqrt{p}\longmapsto i\sqrt[4]{p}$ and $\bar{j}_{p_{\infty}}: \sqrt{p}\longmapsto -i\sqrt[4]{p}$, which are complex embeddings, and $i_{\tilde{p}_{\infty}}$ can be extend to the two $\mathbb{Q}$-embeddings:  $j_{\tilde{p}_{\infty}}: \sqrt{p}\longmapsto \sqrt[4]{p}$ and $\bar{j}_{\tilde{p}_{\infty}}: \sqrt{p}\longmapsto -\sqrt[4]{p}$,  which are real embeddings, then $p_{\infty}$ is the unique infinite prime of $k$ that ramifies in $K$.
\begin{lemma}\label{SHL1}
Keeping  previous hypotheses and notations, then
$$  \left(\frac{-1;\,\delta}{p_{\infty}}\right)= \left(\frac{\varepsilon_{p};\,\delta}{\sqrt{p}}\right)= \left(\frac{ -\varepsilon_{p};\,\delta}{\sqrt{p}}\right)=-1.$$		

\end{lemma}

\begin{proof}~\
\begin{enumerate}[{(1)}]
		\item   Let $i_{p_{\infty}}:\sqrt{p} \longmapsto -\sqrt{p}$ be the complex $\mathbb{Q}$-isomorphism of $k$ corresponding to $p_{\infty}$, then by definition  of Hilbert symbol given in \cite[{Ch II \textsection} 7 Definitions 7.1, p. 195]{G03} and \cite[{Ch II \textsection} 7 Definitions 7.3.1, p. 201]{G03}, we have,
		$$\left(\frac{-1,\,\delta}{p_{\infty}}\right)= i_{p_{\infty}}^{-1}((i_{p_{\infty}}(-1),  i_{p_{\infty}}(\delta)) _{p_{\infty}})
		=i_{p_{\infty}}^{-1}((-1, -\delta) _{p_{\infty}})
		=i_{p_{\infty}}^{-1}(-1)
		=-1.$$
		\item  Since $v_{(\sqrt{p})}(\delta)=1$ and via the property \cite[Lemma V, p. 105]{N5},
		$$\left(\frac{\varepsilon_{p},\,\delta}{\sqrt{p}}\right)=\left[\frac{  \varepsilon_{p}}{\sqrt{p}}\right]^{v(\delta)}=\left[\frac{  \varepsilon_{p}}{\sqrt{p}}\right]=-1.\hspace{16em}$$
		\item  We have 
		$\displaystyle\left(\frac{ -1,\,\delta}{\sqrt{p}}\right)=\left[\frac{ -1}{\sqrt{p}}\right]=\left(\frac{ -1}{p}\right)=1$.
		Then by applying  multiplicative property of  Hilbert symbol we get
		$\displaystyle\left(\frac{-\varepsilon_{p}, \delta}{\sqrt{p}}\right)=\left(\frac{\varepsilon_{p}, \delta}{\sqrt{p}}\right)\left(\frac{-1, \delta}{\sqrt{p}}\right)=-1$.
	\end{enumerate}
\end{proof}
\begin{proof}[\rm\textbf{Proof of Theorem A}]
Firstly, we know from Hasse norm theorem \cite[{Ch II \textsection} 6, Theorem 6.2, p. 179]{G03} that  an element $\alpha$ in $k^{\times}$ is norm in $K$ if and only if $\left(\frac{\alpha,\, \delta}{\mathcal{P}_{i}}\right)=1$ for all primes of $k$ ramified in $K$. So, by Lemma $(\ref{SHL1})$, the units $ -1$, $\varepsilon_{p}$, $ -\varepsilon_{p}$ of $k$ are not norms of elements of $K$. Since
$E_{k}=\langle-1, \varepsilon_{p}\rangle$ and $E_k \cap N_{K/k}(K^\times)=\langle \varepsilon^{2}_{p}\rangle$, therefore
$[E_k:E_k \cap N_{K/k}(K^\times)]=4$, so $e=2$. Secondly, the finite and infinite primes ramified in $K/k$ are
 $p_{\infty}, (\sqrt{p}), 2_{I}, \tilde{q_{1}}, \cdots, \tilde{q_{s}}, \pi_{1}, \cdots, \pi_{t}$, $\tilde{\pi_{1}}, \cdots, \tilde{\pi_{t}}.$ 	Then their number is $ 2t+s+3$. Finally, the $2$-rank of the class group of $K$ is computed using  formula \eqref{ACNF2}.
\end{proof}
\begin{cor}\label{SHC}
Keeping  previous hypotheses and notations, then $ r_{2}(K)=2 $ if and only if one of the following conditions holds:
  		\begin{enumerate}[{\indent\rm (a)}]
  			\item  $d=q$ is a prime number such that $\left(\frac{p}{q}\right)=1$,
  			\item  $d=q_{1}q_{2}$, where $q_{1}, q_{2}$ are odd primes such that $\left(\frac{p }{q_{1}}\right)=\left(\frac{p }{q_{2}}\right)=-1$.
  		\end{enumerate}
  \end{cor}

  \begin{proof}
By Theorem A, $r_{2}(K)=2 $ if and only if ($t=1$ and $s=0$) or ($t=0$ and $s=2$). Hence,
  	\begin{enumerate}[{\indent\rm (a)}]
  		\item  if $t=1$ and $s=0$, then $d=q$ is a prime number and $\left(\frac{p}{q}\right)=1$,
  		\item  if $t=0$ and $s=2$, then $d=q_{1}q_{2}$ where $q_{1}, q_{2}$ are odd primes such that $\left(\frac{p }{q_{1}}\right)=\left(\frac{p }{q_{2}}\right)=-1$.
  	\end{enumerate}
  \end{proof}
\section{\large \bf{The $4$-rank of $ Cl(K)$}\label{hay1}}
In this section, we compute the $4$-rank of $Cl(K)$ whenever $r_2(K)=2$. At the end of each case, we give some numerical examples which are verified using the Pari/GP calculator (version 2-11-3), \cite{PG18}.
\subsection{Case: $d=q$ is a prime number and $\left(\frac{p}{q}\right)=1$}
Let $ p$ and $ q$ be two different odd prime numbers such that $p\equiv 5 \pmod 8$, $\left(\frac{p}{q}\right)=1$ and $\pi_{1}$, $\tilde{\pi}_{1}$ are prime ideals of $\mathcal{O}_k$ lying above $q$,
the ideals $\pi_{1}^{h}$, $\tilde{\pi}_{1}^{h}$ are principal ideals of $\calO_k$, put $\pi_{1}^{h}=(x+y\sqrt{p})/2$  and $\tilde{\pi}_{1}^{h}=(x-y\sqrt{p})/2$, where $h=h(k)$ is
the class number of $k=\mathbb{Q}(\sqrt{p})$. Without loss of generality, we can suppose that  $(x-y\sqrt{p})>0$ and $(x+y\sqrt{p}) >0$, because
\begin{itemize}
	\item [$ \bullet $] if $(x-y\sqrt{p})<0$ and $(x+y\sqrt{p})<0$, we replace $(x+y\sqrt{p})$ by $-(x+y\sqrt{p})$ and $(x-y\sqrt{p})$ by $-(x-y\sqrt{p})$. Which are also positive generators for the ideals $\pi_{1}^{h}$ and $\tilde{\pi}_{1}^{h}$ respectively. 
	\item [$ \bullet $] if $(x-y\sqrt{p})>0$ and $(x+y\sqrt{p})<0$, we have $N_{k/\mathbb{Q}}(\varepsilon_{p})=\varepsilon_{p}\varepsilon'_{p}=-1$ and $\varepsilon'_{p}<0$, then we can replace $(x+y\sqrt{p})$ by $\varepsilon_{p}(x+y\sqrt{p})$ and $(x-y\sqrt{p})$ by $\varepsilon'_{p}(x-y\sqrt{p})$, which are also positive generators for the ideals $\pi_{1}^{h}$, $\tilde{\pi}_{1}^{h}$ respectively.
\end{itemize}	

Furthermore,  applying  absolu norm map to the ideal $\pi_{1}^{h}=(x+y\sqrt{p})/2)$, one gets $\pm4q^{h}=x^{2}-py^{2}$. But if we considre $x+y\sqrt{p}>0 $ and $x-y\sqrt{p}>0 $ this equation becomes $4q^{h}=x^{2}-py^{2}$. So   we can reduce our study to  the case when the ideals $\pi_{1}^{h}$ and $\tilde{\pi}_{1}^{h}$ have positive generators.

  To  compute  the $ 4$-rank of the class group of $K$ by using the generalized R\'edei-matrix, we need the following lemmas.
\begin{lemma}\label{SHL2}
	Keeping  previous hypotheses and notations, then
	\begin{enumerate}[\rm(a)]
		\item $\left(\frac{\varepsilon_{p},\, \delta}{p_{\infty}}\right)=\left(\frac{\sqrt{p},\, \delta}{p_{\infty}}\right)=-1.$
		\item $\left(\frac{2,\, \delta}{p_{\infty}}\right)=\left(\frac{2(x+y\sqrt{p}),\, \delta}{p_{\infty}}\right)=\left(\frac{2(x-y\sqrt{p)},\, \delta}{p_{\infty}}\right)=1.$		
	\end{enumerate}		
\end{lemma}

\begin{proof} By the  definition  of Hilbert symbol given in \cite[{Ch II, \textsection} 7 Definitions 7.1, p. 195]{G03} and \cite[{Ch II \textsection} 7 Definitions 7.3.1, p. 201]{G03}, we have
	\begin{enumerate}[(a)]
		\item
$\left(\frac{\varepsilon_{p},\, \delta}{p_{\infty}}\right)= i_{p_{\infty}}^{-1}((i_{p_{\infty}}(\varepsilon_{p}), i_{p_{\infty}}(\delta)) _{p_{\infty}})
=i_{p_{\infty}}^{-1}((\varepsilon'_{p}, -\delta) _{p_{\infty}})
=i_{p_{\infty}}^{-1}(-1)
=-1,$
since the fundamental unit of $ k$ satisfy $\varepsilon_{p}>0$  and $ N_{k/\mathbb{Q}}(\varepsilon_{p})=\varepsilon_{p}\varepsilon'_{p}=-1$, then $\varepsilon'_{p}<0.$\\
$\left(\frac{\sqrt{p},\, \delta}{p_{\infty}}\right)= i_{p_{\infty}}^{-1}((i_{p_{\infty}}(\sqrt{p}), i_{p_{\infty}}(\delta)) _{p_{\infty}})
=i_{p_{\infty}}^{-1}((-\sqrt{p}, -\delta) _{p_{\infty}})
=i_{p_{\infty}}^{-1}(-1)
=-1.$
	\item 	
$\left(\frac{2,\, \delta}{p_{\infty}}\right)= i_{p_{\infty}}^{-1}((i_{p_{\infty}}(2), i_{p_{\infty}}(\delta)) _{p_{\infty}})
=i_{p_{\infty}}^{-1}((2, -\delta) _{p_{\infty}})
=i_{p_{\infty}}^{-1}(1)
=1$.
$$\left(\frac{2(x+y\sqrt{p}),\, \delta}{p_{\infty}}\right)= i_{p_{\infty}}^{-1}((i_{p_{\infty}}(2(x+y\sqrt{p})), i_{p_{\infty}}(\delta)) _{p_{\infty}})
=i_{p_{\infty}}^{-1}((2(x-y\sqrt{p}), -\delta) _{p_{\infty}})
=i_{p_{\infty}}^{-1}(1)
=1.$$
$$\left(\frac{2(x-y\sqrt{p}),\, \delta}{p_{\infty}}\right)= i_{p_{\infty}}^{-1}((i_{p_{\infty}}(2(x-y\sqrt{p})), i_{p_{\infty}}(\delta)) _{p_{\infty}})
=i_{p_{\infty}}^{-1}((2(x+y\sqrt{p}), -\delta) _{p_{\infty}})
=i_{p_{\infty}}^{-1}(1)
=1.$$
	\end{enumerate}

\end{proof}

\begin{lemma}\label{SHL3} Let $p$ and $q$ be distinct odd primes with $p\equiv 1\pmod 4$. Then if $(x, y)$ is an integer solution of the equation $u^{2}-pv^{2}=4q^{h}$, then
	    		$$  \left(\frac{ x}{p}\right)=\left(\frac{ 2}{p}\right)\left(\frac{ q}{p}\right)_{4}.$$
Moreover, if $p\equiv 5\pmod 8$, then
   $$\left(\frac{ x}{q}\right)= \begin{cases} \left(\frac{p}{q}\right)_{4},& \text{if $ q\equiv 1\pmod 4$, }  \\
   \\
   -\left(\frac{q}{p}\right)_{4}, & \text{if $q\equiv3 \pmod 4$.}
        \end{cases}$$
\end{lemma}

\begin{proof}Let $(x, y)$ be a solution of the equation $4q^{h}=u^{2}-pv^{2}$, then $x^{2}$ and $y^{2}$ have same parity, so $x$ and $y$ have same parity too. Hence $4q^{h}\equiv x^2\pmod p$. As $p\equiv 1\pmod 4$, so 	
	    $$\begin{array}{rcl}
	    \left(\frac{ x}{p}\right)
     	&\equiv&(x^{2})^{\frac{p-1}{4}} \pmod p\\
     	&\equiv& \left(\frac{ x^{2}}{p}\right)_{4} \\
    	&\equiv& \left(\frac{4q^{h}}{p}\right)_{4} \\
    	&\equiv& \left(\frac{ 4}{p}\right)_{4}\left(\frac{ q}{p}\right)_{4} \\
	    &=&  \left(\frac{2}{p}\right)\left(\frac{ q}{p}\right)_{4}.
    	\end{array}
    	$$
    		
For the second equality, put $ x=2^{r}x'$ and $ y=2^{e}y'$ such that $2\nmid x'$ and $2\nmid y'$.
	 \begin{enumerate}[(a)]
	 	\item  If $q\equiv 1 \pmod 4$, then
		$ \begin{array}[t]{ll}
		\left(\frac{ x}{q}\right)&= \left(\frac{x^{2}}{q}\right)_{4}=\left(\frac{y^{2}p}{q}\right)_{4}\\
		&=\left(\frac{y}{q}\right)\left(\frac{p}{q}\right)_{4}=\left(\frac{2}{q}\right)^{e}\left(\frac{q}{y'}\right)\left(\frac{p}{q}\right)_{4}\\
        &=\left(\frac{ 2}{q}\right)^{e}\left(\frac{4q^{h}}{y'}\right)\left(\frac{p}{q}\right)_{4}\\
        &=\left(\frac{2}{q}\right)^{e}\left(\frac{p}{q}\right)_{4}.
        \end{array} $
	 	\begin{itemize}
	 	\item [$ \bullet $] If $q\equiv 1\pmod 8$ or $y$ is odd, then $ \left(\frac{ x}{q}\right)=\left(\frac{p}{q}\right)_{4}$.
	   	\item [$ \bullet $] If $q\equiv 5\pmod 8$ and $x$, $y$ are even, then $x=2a$ and $y=2b$, so $q^{h}=a^{2}-pb^{2}$. Hence by calculation modulo $8$ and taking into account the fact that $h$ is odd,  we get the following table: \\
	 	\begin{center}
	 		
	 		\begin{tabular}{c|rrrrrrrr}
	 			\hline\hline
	 			$a$ & 0 & 1& 2& 3& 4 & 5 & 6& 7 \\
	 			\hline
	 			$a^{2}$ & 0 & 1 & 4& 1& 0 & 1 & 4& 1 \\
	 			\hline
	 			$5-a^{2}$ & 5 & 4 & 1& 4& 5& 4& 1& 4\\
	 			\hline\hline
	 		\end{tabular}
	 		\label{tab:hresult}
	 	\end{center}

	 To this end, as $b^2\equiv0$, $1$ or $4\pmod8$, so $3b^{2}\equiv 0, 3$ or $4\pmod 8$. On the other hand, since $q^h-a^2\equiv 5-a^2\equiv 3b^2\pmod8$ one deduces that $3b^{2}\equiv 4\pmod 8$, thus  multiplying by the inverse of $3\pmod8$ we get $b^{2}\equiv 4 \pmod 8$. Hence $e=2$, this in turn yields that
	 	$$\left(\frac{ x}{q}\right)=\left(\frac{p}{q}\right)_{4}.$$
	 \end{itemize}	
	  	
	 	\item If $q\equiv 3 \pmod 4$, then by  quadratic reciprocity low we get
	 		 $$  \begin{array}{ll}
	 		 \left(\frac{ x}{q}\right)= \left(\frac{2^{r}x'}{q}\right)
	 		 =\left(\frac{2}{q}\right)^{r}\left(\frac{-q}{x'}\right)
	 		 =\left(\frac{2}{q}\right)^{r}\left(\frac{-4q^{h}}{x'}\right)
	 		 =\left(\frac{2}{q}\right)^{r}\left(\frac{p}{x'}\right)=\left(\frac{2}{q}\right)^{r}\left(\frac{2}{p}\right)^{r}\left(\frac{x}{p}\right)=\left(\frac{2}{q}\right)^{r}\left(\frac{q}{p}\right)_{4}(-1)^{r+1}.
	 		 \end{array}$$		
	 		\begin{itemize}
	 			\item [$ \bullet $] If $q\equiv 3\pmod 8$ or $x$ is odd, then $ \left(\frac{ x}{q}\right)= -\left(\frac{q}{p}\right)_{4}$.
	 			\item [$ \bullet $] If $q\equiv 7\pmod 8$ and $x$ and $y$ are even  $x=2a$ and $y=2b$, so $q^{h}=a^{2}-pb^{2}$, hence by calculation modulo 8 we get the following table: \\
	 			\begin{center}
	 				\begin{tabular}{c|rrrrrrrr}
	 					\hline\hline
	 					$b$ & 0 & 1& 2& 3& 4 & 5 & 6& 7\\
	 					\hline
	 					$b^{2}$ & 0 & 1 & 4& 1& 0 & 1 & 4& 1\\
	 					\hline
	 					$7+5b^{2}$ & 7 & 4 & 3& 4& 7& 4& 3& 4\\
	 					\hline\hline
	 				\end{tabular}
	 				\label{tab:hresult}
	 			\end{center}\  \\
	 			with the following congruences :  $b^2\equiv0$, $1$ or $4\pmod8$ and  $a^2\equiv q^h+pb^2\equiv 7+5b^{2}$,we find that $a^{2}\equiv 4 \pmod 8$, then $r=2$. Which also means that
	 			$$\left(\frac{ x}{q}\right)=-\left(\frac{q}{p}\right)_{4}.$$
	 			\end{itemize}
	 		
	 \end{enumerate}
\end{proof}

\begin{lemma}\label{SHL4} Let $ p$ and $ q$ be two different odd prime numbers such that $p\equiv 5 \pmod 8$ and $\left(\frac{p}{q}\right)=1$, then
\begin{enumerate}[\indent\rm(a)]
    	\item  $\left(\frac{2,\, \delta}{\sqrt{p}}\right)=-1.$	
        \item $\left(\frac{\sqrt{p},\, \delta}{\sqrt{p}}\right)=-\left(\frac{\varepsilon_{p},\, \delta}{\sqrt{p}}\right)=1.$
         \\	    		
		\item  $\left(\frac{2(x+y\sqrt{p}),\,  \delta}{\sqrt{p}}\right)=\left(\frac{2(x-y\sqrt{p}),\,  \delta}{\sqrt{p}}\right)=\left(\frac{ q}{p}\right)_{4}.$
\end{enumerate}
\end{lemma}

\begin{proof}We have 
\begin{enumerate}[(a)]
		\item $\left(\frac{ 2,\, \delta}{\sqrt{p}}\right)= \left[\frac{ 2}{\sqrt{p}}\right]=\left(\frac{ 2}{p}\right)	=-1.$
		 \\
		\item
		$\left(\frac{\sqrt{p},\, \delta}{\sqrt{p}}\right)=\left[\frac{-q}{\sqrt{p}}\right]=\left(\frac{ -1}{p}\right)\left(\frac{ q}{p}\right)
		=\left(\frac{p}{q}\right)= 1.$
		\item
		$\left(\frac{2(x+y\sqrt{p}),\,  \delta}{\sqrt{p}}\right)=\left[\frac{  2(x+y\sqrt{p})}{\sqrt{p}}\right]=\left(\frac{2x}{p}\right)
		=\left(\frac{2}{p}\right)\left(\frac{x}{p}\right)=\left(\frac{q}{p}\right)_{4}.$
		A similar argument shows that
		$\left(\frac{2(x-y\sqrt{p}),\,  \delta}{\sqrt{p}}\right)=\left(\frac{ q}{p}\right)_{4}.$			
	\end{enumerate}			

\end{proof}

\begin{lemma}
	Keep   hypotheses and notations of Lemma $\ref{SHL4}$, then
	\begin{enumerate}[\indent\rm(a)]
		\item $\left(\frac{-1,\, \delta}{\pi_{1}}\right)=\left(\frac{ -1,\, 
			\delta}{\tilde{\pi_{1}}}\right)=\begin{cases} 1, & \text{if $q\equiv1 \pmod 4$,} \\
		-1, & \text{if $q\equiv3 \pmod 4$.}
		\end{cases}$
		\item
		\begin{itemize}
		\item  If $q\equiv 1 \pmod 4$, then $\left(\frac{ \varepsilon_{p},\, \delta}{\pi_{1}}\right)=\left(\frac{ \varepsilon_{p},\, \delta}{\tilde{\pi}_{1}}\right)= \left(\frac{ p}{q}\right)_{4}\left(\frac{ q}{p}\right)_{4}.$
		\\
		\item 	If $q\equiv 3 \pmod 4$, then $\left(\frac{\varepsilon_{p},\, \delta}{\pi_{1}}\right)=-\left(\frac{\varepsilon_{p},\, \delta}{\tilde{\pi_{1}}}\right)=-\left(\frac{y}{q}\right)$.
		\\
		\end{itemize}
		\item $ \left(\frac{2,\, \delta}{\pi_{1}}\right)=\left(\frac{2,\, \delta}{\tilde{\pi}_{1}}\right)=\left(\frac{ 2}{q}\right).$
		\item
		\begin{itemize}
			\item  If $q\equiv 1 \pmod 4$, then $\left(\frac{\sqrt{p},\, \delta}{\pi_{1}}\right)=\left(\frac{\sqrt{p},\, \delta}{\tilde{\pi}_{1}}\right)= \left(\frac{ p}{q}\right)_{4}.$
			\\
			\item 	If $q\equiv 3 \pmod 4$, then $\left(\frac{\sqrt{p},\, \delta}{\pi_{1}}\right)=-\left(\frac{\sqrt{p},\, \delta}{\tilde{\pi}_{1}}\right)=\left(\frac{y}{q}\right)\left(\frac{q}{p}\right)_{4}.$
			\\
		\end{itemize}
		\item
		\begin{itemize}
			\item  If $q\equiv 1 \pmod 4$, then $\left(\frac{2(x+y\sqrt{p}),\, \delta}{\pi_{1}}\right)=\left(\frac{2(x-y\sqrt{p}),\, \delta}{\tilde{\pi}_{1}}\right)=1.$
			\\
			\item 	If $q\equiv 3 \pmod 4$, then $  \left(\frac{2(x+y\sqrt{p}),\, \delta}{\pi_{1}}\right)=-\left(\frac{2(x-y\sqrt{p}),\, \delta}{\tilde{\pi}_{1}}\right)=\left(\frac{y}{q}\right)$.
			\\
		\end{itemize}
		\item $  \left(\frac{2(x-y\sqrt{p}),\,  \delta}{\pi_{1}}\right)=\left(\frac{2(x+y\sqrt{p}),\,  \delta}{\tilde{\pi}_{1}}\right)=\begin{cases} \left(\frac{p}{q}\right)_{4}, & \text{if $ q\equiv 1 \pmod 4$, } \\
		\\
		-\left(\frac{q}{p}\right)_{4}, & \text{if $q\equiv 3 \pmod 4$.}
		\end{cases}$
	\end{enumerate}	
\end{lemma}
\begin{proof}The main tools to prove these equalities are the  properties  cited in \cite[{Ch II \textsection} 7 Proposition 7.4.3, p. 205]{G03} and by applying the bilinear property  of Hilbert symbol, we get	
\begin{enumerate}[(a)]	
		\item $
		\left(\frac{ -1,\, \delta}{\tilde{\pi_{1}}}\right)=\left(\frac{ -1,\, \delta}{\pi_{1}}\right)= \left[\frac{ -1}{\pi_{1}}\right]
		=\left(\frac{ -1}{q}\right)=\begin{cases} 1, & \text{if $q\equiv1 \pmod 4$,} \\
		-1,& \text{if $q\equiv3 \pmod 4$.}
		\end{cases}$
		\item
		\begin{itemize}
			\item  	Suppose $q\equiv 1 \pmod 4$. Since $q\equiv p\equiv 1 \pmod 4$ and $\left(\frac{ p}{q}\right)=1$, then by Scholz's reciprocity law ( \cite[{Ch 5 \textsection} 5.2 Proposition 5.8, p. 160]{L00}), we have\\
			$
			\left(\frac{ \varepsilon_{p},\, \delta}{\pi_{1}}\right)= \left[\frac{  \varepsilon_{p}}{\pi_{1}}\right]=\left(\frac{ p}{q}\right)_{4}\left(\frac{ q}{p}\right)_{4}$, moreover $\left[\frac{  \varepsilon_{p}}{\pi_{1}}\right]\left[\frac{  \varepsilon_{p}}{\tilde{\pi_{1}}}\right] =\left(\frac{ -1}{q}\right)=1$, then $  \left(\frac{ \varepsilon_{p},\, \delta}{\tilde{\pi_{1}}}\right)=\left(\frac{ p}{q}\right)_{4}\left(\frac{ q}{p}\right)_{4}.$
			\item 	If $q\equiv 3 \pmod 4$ we find that $\left(\frac{\varepsilon_{p},\, \delta}{\pi_{1}}\right)= \left[\frac{  \varepsilon_{p}}{\pi_{1}}\right]$ and $\left(\frac{\varepsilon_{p},\, \delta}{\tilde{\pi_{1}}}\right)= \left[\frac{  \varepsilon_{p}}{\tilde{\pi_{1}}}\right]$. \\
			Consequently, $\left[\frac{  \varepsilon_{p}}{\pi_{1}}\right]\left[\frac{  \varepsilon_{p}}{\tilde{\pi_{1}}}\right]=\left[\frac{  \varepsilon_{p}}{\tilde{\pi_{1}}\pi_{1}}\right]=\left[\frac{  \varepsilon_{p}}{qO_{k}}\right] =\left(\frac{N_{k/\mathbb{Q}}(\varepsilon_{p})}{q}\right)=\left(\frac{-1}{q}\right)=-1$. \\
			As $\left[\frac{  \varepsilon_{p}\sqrt{p}}{\pi_{1}}\right]=-\left(\frac{ q}{p}\right)_{4}$ (via \cite{W74}) and $\left[\frac{\varepsilon_{p}\sqrt{p}}{\pi_{1}}\right]=\left[\frac{  -\varepsilon_{p}}{\pi_{1}}\right]\left[\frac{ -\sqrt{p}}{\pi_{1}}\right]=\left(\frac{-1}{q}\right)\left[\frac{  \varepsilon_{p}}{\pi_{1}}\right]\left(\frac{y}{q}\right)\left(\frac{2}{p}\right)\left(\frac{q}{p}\right)_{4}$.\\
			Then $\left[\frac{\varepsilon_{p}}{\pi_{1}}\right]=-\left(\frac{-1}{q}\right)\left(\frac{2}{p}\right)\left(\frac{y}{q}\right)=\left(\frac{2}{p}\right)\left(\frac{y}{q}\right)$.
			This completes the proof.
			\\
		\end{itemize}
		\item 	
		$
		\left(\frac{2,\, \delta}{\pi_{1}}\right)= \left[\frac{ 2}{\pi_{1}}\right]
		=\left(\frac{ 2}{q}\right)$ and
		$\left(\frac{2,\, \delta}{\tilde{\pi_{1}}}\right)= \left[\frac{  2}{\tilde{\pi_{1}}}\right]=\left(\frac{ 2}{q}\right).$\\	
		\item
		\begin{itemize}
			\item  	If $q\equiv 1 \pmod 4$, then\\
			$\left(\frac{\sqrt{p},\, \delta}{\pi_{1}}\right)
			=\left[\frac{\sqrt{p}}{\pi_{1}^{h}}\right]
			=\left[\frac{2y}{\pi_{1}^{h}}\right]\left[\frac{2x-2(x+y\sqrt{p})}{\pi_{1}^{h}}\right]= \left[\frac{2y}{\pi_{1}}\right] \left[\frac{2x}{\pi_{1}}\right]= \left(\frac{y}{q}\right)\left(\frac{x}{q}\right)
			=\left(\frac{p}{q}\right)_{4},
			$
			\\
			and
			
			$ \left(\frac{\sqrt{p},\, \delta}{\tilde{\pi_{1}}}\right)=\left[\frac{\sqrt{p}}{\tilde{\pi_{1}}^{h}}\right]=\left[\frac{2y}{\tilde{\pi_{1}}^{h}}\right]\left[\frac{-2x+2(x-y\sqrt{p})}{\tilde{\pi_{1}}^{h}}\right]= \left[\frac{2y}{\tilde{\pi_{1}}}\right] \left[\frac{-2x}{\tilde{\pi_{1}}}\right]=\left(\frac{y}{q}\right)\left(\frac{ -1}{q}\right)\left(\frac{x}{q}\right)=\left(\frac{p}{q}\right)_{4}.$\\
			\item  In the same way, we show that if $q\equiv 3 \pmod 4$, then\\
			$
			\left(\frac{\sqrt{p},\,  \delta}{\pi_{1}}\right)
			= \left[\frac{\sqrt{p}}{\pi_{1}^{h}}\right]= \left[\frac{2y}{\pi_{1}^{h}}\right] \left[\frac{2y\sqrt{p}}{\pi_{1}^{h}}\right]=\left[\frac{2y}{\pi_{1}}\right] \left[\frac{-2x}{\tilde{\pi_{1}}}\right]=\left(\frac{-1}{q}\right)\left(\frac{y}{q}\right)\left(\frac{x}{q}\right)=
           \left(\frac{y}{q}\right)\left(\frac{q}{p}\right)_{4},$\\
           \\
			$
			\left(\frac{\sqrt{p},\, \delta}{\tilde{\pi_{1}}}\right)=\left[\frac{\sqrt{p}}{\tilde{\pi_{1}}^{h}}\right]= \left[\frac{2y}{\tilde{\pi_{1}}^{h}}\right] \left[\frac{2y\sqrt{p}}{\tilde{\pi_{1}}^{h}}\right]= \left[\frac{2y}{\tilde{\pi_{1}}}\right] \left[\frac{2x}{\tilde{\pi_{1}}}\right]= \left(\frac{y}{q}\right)\left(\frac{x}{q}\right)=-\left(\frac{y}{q}\right)\left(\frac{q}{p}\right)_{4}.$
		\end{itemize}
		\item
		\begin{itemize}
			\item  	We also have that if $q\equiv 1 \pmod 4$, then\\
				$\begin{array}{rcl}
				\left(\frac{2(x+y\sqrt{p}),\,  \delta}{\pi_{1}}\right)
				&=&\left(\frac{2(x+y\sqrt{p}),\,  \delta}{\pi_{1}}\right)\left(\frac{-\sqrt{p}(x-y\sqrt{p}),\,  \delta}{\pi_{1}}\right)\left(\frac{-\sqrt{p}(x-y\sqrt{p}),\,  \delta}{\pi_{1}}\right)\\
				&=&\left(\frac{-8q^{h}\sqrt{p},\,  \delta}{\pi_{1}}\right)\left(\frac{-\sqrt{p}(x-y\sqrt{p}),\,  \delta}{\pi_{1}}\right)\\
				&=&\left(\frac{-\delta,\,  \delta}{\pi_{1}}\right)\left(\frac{-2\sqrt{p}(x-y\sqrt{p}),\,  \delta}{\pi_{1}}\right)
				=\left[\frac{-2\sqrt{p}(x-y\sqrt{p})}{\pi_{1}}\right]  \\
				&=&\left[\frac{-2\sqrt{p}(x-y\sqrt{p})}{\pi_{1}^{h}}\right]
				=\left[\frac{-\sqrt{p}}{\pi_{1}^{h}}\right]\left[\frac{2(x-y\sqrt{p})}{\pi_{1}^{h}}\right] \\
				&=&\left[\frac{-\sqrt{p}}{\pi_{1}^{h}}\right]\left[\frac{4x-2(x+y\sqrt{p})}{\pi_{1}^{h}}\right]	
				=\left[\frac{-\sqrt{p}}{\pi_{1}^{h}}\right]\left[\frac{4x}{\pi_{1}^{h}}\right]\\	
				&=&\left[\frac{-\sqrt{p}}{\pi_{1}^{h}}\right]\left[\frac{x}{\pi_{1}}\right]	
				=\left(\frac{p}{q}\right)_{4}\left(\frac{ x}{q}\right)
				=\left(\frac{p}{q}\right)_{4}\left(\frac{p}{q}\right)_{4}
				=1.
				\end{array}$	
			
				$\begin{array}{rcl}
				\left(\frac{2(x-y\sqrt{p}),\,  \delta}{\tilde{\pi_{1}}}\right)
				&=&\left(\frac{2(x-y\sqrt{p}),\,  \delta}{\tilde{\pi_{1}}}\right)\left(\frac{-\sqrt{p}(x+y\sqrt{p}),\, \delta}{\tilde{\pi_{1}}}\right)\left(\frac{-\sqrt{p}(x+y\sqrt{p}),\,  \delta}{\tilde{\pi_{1}}}\right)\\
				&=&\left(\frac{-\delta,\,  \delta}{\tilde{\pi_{1}}}\right)\left(\frac{-2\sqrt{p}(x+y\sqrt{p}),\,  \delta}{\tilde{\pi_{1}}}\right)
				=\left[\frac{-2\sqrt{p}(x+y\sqrt{p})}{\tilde{\pi_{1}}}\right]  \\
				&=&\left[\frac{-\sqrt{p}}{\tilde{\pi_{1}}}\right]\left[\frac{2(x+y\sqrt{p})}{\tilde{\pi_{1}}}\right]
				=\left[\frac{-\sqrt{p}}{\tilde{\pi_{1}}^{h}}\right]\left[\frac{2(x+y\sqrt{p})}{\tilde{\pi_{1}}^{h}}\right]
				=\left[\frac{-\sqrt{p}}{\tilde{\pi_{1}}^{h}}\right]\left[\frac{4x-2(x-y\sqrt{p})}{\tilde{\pi_{1}}^{h}}\right]\\	
				&=&\left[\frac{-\sqrt{p}}{\tilde{\pi_{1}}^{h}}\right]\left[\frac{4x}{\tilde{\pi_{1}}^{h}}\right]	
				=\left[\frac{-\sqrt{p}}{\tilde{\pi_{1}}^{h}}\right]\left[\frac{x}{\tilde{\pi_{1}}}\right]\\	
				&=&\left(\frac{p}{q}\right)_{4}\left(\frac{ x}{q}\right)
				=\left(\frac{p}{q}\right)_{4}\left(\frac{p}{q}\right)_{4}=1.
				\end{array}$	
			\item If $q\equiv 3 \pmod 4$, then\\
			$\begin{array}{rcl}
			\left(\frac{2(x+y\sqrt{p}),\,  \delta}{\pi_{1}}\right)
			&=&\left(\frac{2(x+y\sqrt{p}),\,  \delta}{\pi_{1}}\right)\left(\frac{-2\sqrt{p}(x-y\sqrt{p}),\,  \delta}{\pi_{1}}\right)\left(\frac{-2\sqrt{p}(x-y\sqrt{p}),\,  \delta}{\pi_{1}}\right)\\
			&=&\left(\frac{-4q^{h}\sqrt{p},\,  \delta}{\pi_{1}}\right)\left(\frac{-2\sqrt{p}(x-y\sqrt{p}),\,  \delta}{\pi_{1}}\right)\\
			&=&\left(\frac{-\delta,\,  \delta}{\pi_{1}}\right)\left(\frac{-2\sqrt{p}(x-y\sqrt{p}),\,  \delta}{\pi_{1}}\right)\\
			&=&\left[\frac{-2\sqrt{p}(x-y\sqrt{p})}{\pi_{1}}\right] =\left[\frac{-2\sqrt{p}(x-y\sqrt{p})}{\pi_{1}^{h}}\right]  \\			&=&\left[\frac{-\sqrt{p}}{\pi_{1}^{h}}\right]\left[\frac{2(x-y\sqrt{p})}{\pi_{1}^{h}}\right]=
            \left[\frac{-\sqrt{p}}{\pi_{1}^{h}}\right]\left[\frac{4x-2(x+y\sqrt{p})}{\pi_{1}^{h}}\right]\\	
			&=&\left[\frac{-\sqrt{p}}{\pi_{1}^{h}}\right]\left[\frac{4x}{\pi_{1}^{h}}\right]= \left[\frac{-\sqrt{p}}{\pi_{1}^{h}}\right]\left[\frac{x}{\pi_{1}}\right]\\	
			&=&\left(\frac{y}{q}\right)\left(\frac{x}{q}\right)\left(\frac{x}{q}\right)=\left(\frac{y}{q}\right).
			\end{array}$\\	
				
			$\begin{array}{rcl}
			\left(\frac{2(x-y\sqrt{p}),\,  \delta}{\tilde{\pi_{1}}}\right)
			&=&\left(\frac{2(x-y\sqrt{p}),\,  \delta}{\tilde{\pi_{1}}}\right)\left(\frac{-2\sqrt{p}(x+y\sqrt{p}),\,  \delta}{\tilde{\pi_{1}}}\right)\left(\frac{-2\sqrt{p}(x+y\sqrt{p}),\,  \delta}{\tilde{\pi_{1}}}\right)\\
			&=&\left(\frac{-\delta,\,  \delta}{\tilde{\pi_{1}}}\right)\left(\frac{-2\sqrt{p}(x+y\sqrt{p}),\,  \delta}{\tilde{\pi_{1}}}\right)
			=\left[\frac{-2\sqrt{p}(x+y\sqrt{p})}{\tilde{\pi_{1}}}\right]  \\
			&=&\left[\frac{-\sqrt{p}}{\tilde{\pi_{1}}}\right]\left[\frac{2(x+y\sqrt{p})}{\tilde{\pi_{1}}}\right]
			=\left[\frac{-\sqrt{p}}{\tilde{\pi_{1}}^{h}}\right]\left[\frac{2(x+y\sqrt{p})}{\tilde{\pi_{1}}^{h}}\right] \\
			&=&\left[\frac{x}{\tilde{\pi_{1}}^{h}}\right]\left[\frac{4x-2(x-y\sqrt{p})}{\tilde{\pi_{1}}^{h}}\right]	
			=\left[\frac{-\sqrt{p}}{\tilde{\pi_{1}}^{h}}\right]\left[\frac{4x}{\tilde{\pi_{1}}^{h}}\right]\\	
			&=&\left[\frac{-\sqrt{p}}{\tilde{\pi_{1}^{h}}}\right]\left[\frac{x}{\tilde{\pi_{1}}}\right]
			=-\left(\frac{y}{q}\right)\left(\frac{x}{q}\right)\left(\frac{x}{q}\right)
			=-\left(\frac{y}{q}\right).
			\end{array}$	
		\end{itemize}
	
	\item 	With similar calculations we prove that
          \begin{itemize}
          	\item  	If $q\equiv 1 \pmod 4$, then\\
			$\left(\frac{2(x+y\sqrt{p}),\, \delta}{\tilde{\pi_{1}}}\right) =\left[\frac{2(x+y\sqrt{p})}{\tilde{\pi_{1}}}\right]=\left[\frac{4x}{\tilde{\pi_{1}}}\right]=\left(\frac{x}{q}\right)=\left(\frac{p}{q}\right)_{4},$\\		
			and
			
			$\left(\frac{2(x-y\sqrt{p}),\, \delta}{\pi_{1}}\right)=\left[\frac{2(x-y\sqrt{p})}{\pi_{1}}\right]=\left[\frac{4x}{\pi_{1}}\right]=\left(\frac{x}{q}\right)=\left(\frac{p}{q}\right)_{4}.$
		\item If $q\equiv 3 \pmod 4$, then\\
		$  \left(\frac{2(x+y\sqrt{p}),\, \delta}{\tilde{\pi_{1}}}\right)=\left[\frac{2(x+y\sqrt{p})}{\tilde{\pi_{1}}}\right] =\left[\frac{4x}{\tilde{\pi_{1}}}\right]=\left(\frac{ x}{q}\right)=\left(\frac{2}{p}\right)\left(\frac{q}{p}\right)_{4}=-\left(\frac{q}{p}\right)_{4},$	\\	
		   and
		
		$
		   \left(\frac{2(x-y\sqrt{p}),\, \delta}{\pi_{1}}\right)=\left[\frac{2(x-y\sqrt{p})}{\pi_{1}}\right] =\left[\frac{4x}{\pi_{1}}\right]=\left(\frac{ x}{q}\right)=\left(\frac{2}{p}\right)\left(\frac{q}{p}\right)_{4}=-\left(\frac{q}{p}\right)_{4}.$
		\end{itemize}	
	\end{enumerate}	
\end{proof}

\begin{lemma}
	Keeping  previous hypotheses and notations, then
	\begin{enumerate}[\rm(a)]
		\item $\left(\frac{\varepsilon_{p},\, \delta}{2_{I}}\right)=\begin{cases} 1, & \text{if $ q\equiv 1 \pmod 4$}, \\
		-1, & \text{if $ q\equiv 3 \pmod 4$.}
		\end{cases}$
        \item $\left(\frac{2,\, \delta}{2_{I}}\right)=\left(\frac{-1,\, \delta}{2_{I}}\right)=-1,$ and $\left(\frac{ \sqrt{p},\, \delta}{2_{I}}\right)=\begin{cases} -1, & \text{if $ q\equiv 1 \pmod 4$}, \\
        1, & \text{if $ q\equiv 3 \pmod 4$.}
        \end{cases}$	
        \item
    	\begin{itemize}
    	\item  If $q\equiv 1 \pmod 4$, then $\left(\frac{2(x+y\sqrt{p}),\, \delta}{2_{I}}\right)=\left(\frac{2(x-y\sqrt{p}),\, \delta}{2_{I}}\right)= \left(\frac{ p}{q}\right)_{4}\left(\frac{ q}{p}\right)_{4}.$
    		\\
    		
    	\item 	If $q\equiv 3 \pmod 4$, then $  \left(\frac{2(x+y\sqrt{p}),\, \delta}{2_{I}}\right)=-\left(\frac{2(x-y\sqrt{p}),\, \delta}{2_{I}}\right)=-\left(\frac{y}{q}\right)$.
    	\end{itemize}		
	\end{enumerate}		
\end{lemma}

\begin{proof}We prove these results by using  the product formula for  Hilbert symbol and previous lemmas.	
\end{proof}
\begin{remark} The above results can be summered in the following tables. For  $q\equiv 1 \pmod 4$, we have
     \begin{center}
     	\begin{tabular}{||c||c|c|c|c|c||}
     		\hline\hline
     		\diagbox{element}{prime} &$\pi_{1}$ &$\tilde{\pi_{1}}$&$2_{I}$&$(\sqrt{p})$ &$ p_{\infty}$\\
     		\hline\hline
     		$ 2(x+y\sqrt{p})$ &1 &$\left(\frac{p}{q}\right)_{4}$ &$\left(\frac{ p}{q}\right)_{4}\left(\frac{ q}{p}\right)_{4}$ &$\left(\frac{ q}{p}\right)_{4}$ & 1\\
     		\hline
     		$ 2(x-y\sqrt{p})$ &$\left(\frac{p}{q}\right)_{4}$ &1 &$\left(\frac{ p}{q}\right)_{4}\left(\frac{ q}{p}\right)_{4}$  &$\left(\frac{ q}{p}\right)_{4}$ &1 \\
     		\hline
     		$ 2$ &$\left(\frac{ 2}{q}\right)$ &$\left(\frac{ 2}{q}\right)$  & -1 & -1 &1 \\
     		\hline
     		$\sqrt{p}$ &$\left(\frac{p}{q}\right)_{4}$&$\left(\frac{p}{q}\right)_{4}$ & -1& 1  & -1 \\
     		\hline
     		-1 &1 &1&-1& 1 & -1\\
     		\hline
     		
     		$\varepsilon_{p}$ &$\left(\frac{ p}{q}\right)_{4}\left(\frac{ q}{p}\right)_{4}$  &$\left(\frac{ p}{q}\right)_{4}\left(\frac{ q}{p}\right)_{4}$ & 1 & -1 &  -1\\
     		\hline\hline  			  			  			
     	\end{tabular}
     \end{center}
 For  $q\equiv 3 \pmod 4$, we have $(p^{\frac{(q+1)}{8}})^{4}\equiv p^{\frac{(q+1)}{2}}\equiv p^{\frac{(q-1+2)}{2}}\equiv p^{\frac{(q-1)}{2}+1}\equiv \left(\frac{ p}{q}\right)p \equiv p\pmod q$, then $p^{\frac{(q+1)}{8}}$ is a solution to $x^{4}\equiv p \pmod q$, hence $\left(\frac{ p}{q}\right)_{4}=1$.
\begin{center}
    	\begin{tabular}{||c||c|c|c|c|c||}
    		\hline\hline
    		\diagbox{element}{prime} &$\pi_{1}$ &$\tilde{\pi_{1}}$&$2_{I}$&$(\sqrt{p})$ &$ p_{\infty}$\\
    		\hline\hline
    		$ 2(x+y\sqrt{p})$ &$\left(\frac{y}{q}\right)$ &$-\left(\frac{q}{p}\right)_{4}$ &$-\left(\frac{y}{q}\right)$ &$\left(\frac{ q}{p}\right)_{4}$ & 1\\
    		\hline
    		$ 2(x-y\sqrt{p})$ &$-\left(\frac{q}{p}\right)_{4}$ &$-\left(\frac{y}{q}\right)$ &$\left(\frac{y}{q}\right)$  &$\left(\frac{ q}{p}\right)_{4}$ &1 \\
    		\hline
    		$ 2$ &$\left(\frac{ 2}{q}\right)$ &$\left(\frac{ 2}{q}\right)$  & -1 & -1 &1 \\
    		\hline
    		$\sqrt{p}$ &$\left(\frac{y}{q}\right)\left(\frac{q}{p}\right)_{4}$ &$- \left(\frac{y}{q}\right)\left(\frac{q}{p}\right)_{4}$ & 1& 1  & -1 \\
    		\hline
    		-1 &-1 &-1&-1& 1 & -1\\
    		\hline
    		
    		$\varepsilon_{p}$ &$-\left(\frac{y}{q}\right)$  &$\left(\frac{y}{q}\right)$ & -1 & -1 &  -1\\
    		\hline\hline  			  			  			
    	\end{tabular}
\end{center}

\end{remark}

  \begin{proof}[\rm\textbf{Proof of Theorem B}]~\\
  	By assuming  $q$ is a prime number satisfying  $\left(\frac{p}{q}\right)=1$, the  finite primes of $k$ ramifying in $K$ are:  $\pi_{1}$, $\tilde{\pi}_{1}$, $ 2_{I}$ and $(\sqrt{p})$.
  	So there exist prime ideals $\mathcal{H}_{1}, \mathcal{H}_{2}, \mathcal{H}_{3}, \mathcal{H}_{4}$ of $K$  such that
  	$$\mathcal{H}_{1}^{2}=\pi_{1}O_{K}, \mathcal{H}_{2}^{2}=\tilde{\pi}_{1}O_{K}, \mathcal{H}_{3}^{2}=2_{I}O_{K}, \mathcal{H}_{4}^{2}=\sqrt{p}O_{K}. $$
  	
   Therefore, we have

\begin{itemize}
	\item $\mathcal{H}_{1}^{h}\sigma(\mathcal{H}_{1}^{h})=\mathcal{H}_{1}^{h}\mathcal{H}_{1}^{h}O_{K}=\pi_{1}^{h}O_{K}=((x+y\sqrt{p})/2)O_{K}, $
	\item $\mathcal{H}_{2}^{h}\sigma(\mathcal{H}_{2}^{h})=\mathcal{H}_{2}^{h}\mathcal{H}_{2}^{h}O_{K}=\tilde{\pi}_{1}^{h}O_{K}=((x-y\sqrt{p})/2)O_{K}, $
	\item  $\mathcal{H}_{3}\sigma(\mathcal{H}_{3})=\mathcal{H}_{3}\mathcal{H}_{3}O_{K}=2_{I}O_{K}=2O_{K}, $
	\item  $\mathcal{H}_{4}\sigma(\mathcal{H}_{4})=\mathcal{H}_{4}\mathcal{H}_{4}O_{K}=\sqrt{p}O_{K}=\sqrt{p}O_{K}. $
\end{itemize}
where $\sigma$ is the generator of the Galois group of $K/k$, and $h$ the class number of $k$.\\

Moreover, we know from Lemma $\ref{SHL1}$ that $-1$, $\varepsilon_{p}$, $-\varepsilon_{p}\notin E_{k}\cap N_{K/k}(K)$.
Then the $ 5\times(4 + 2)$ generalized R\'edei's matrix of Hilbert symbols \eqref{RMG}, can be written as
$$ R_{K/k} = \begin{pmatrix}
\scriptstyle \left(\frac{2(x+y\sqrt{p}),\, \delta}{\pi_{1}}\right)& \scriptstyle \left(\frac{2(x-y\sqrt{p}),\, \delta}{\pi_{1}}\right)& \scriptstyle \left(\frac{ 2,\, \delta}{\pi_{1}}\right)& \scriptstyle \left(\frac{ \sqrt{p},\, \delta}{\pi_{1}}\right)&
\scriptstyle \left(\frac{ -1,\, \delta}{\pi_{1}}\right)&
\scriptstyle \left(\frac{ \varepsilon_{p},\, \delta}{\pi_{1}}\right) \\
\scriptstyle \left(\frac{2(x+y\sqrt{p}),\, \delta}{\tilde{\pi}_{1}}\right)& \scriptstyle \left(\frac{ 2(x-y\sqrt{p}),\, \delta}{\tilde{\pi}_{1}}\right)& \scriptstyle \left(\frac{ 2,\, \delta}{\tilde{\pi}_{1}}\right)& \scriptstyle \left(\frac{ \sqrt{p},\, \delta}{\tilde{\pi}_{1}}\right)& \scriptstyle \left(\frac{ -1,\, \delta}{\tilde{\pi}_{1}}\right)& \scriptstyle \left(\frac{\varepsilon_{p},\, \delta}{\tilde{\pi}_{1}}\right)  \\
\scriptstyle \left(\frac{2(x+y\sqrt{p}),\, \delta}{2_{I}}\right)&
\scriptstyle \left(\frac{ 2(x-y\sqrt{p}),\, \delta}{2_{I}}\right) &
\scriptstyle \left(\frac{ 2,\, \delta}{2_{I}}\right)&
\scriptstyle \left(\frac{ \sqrt{p},\, \delta}{2_{I}}\right) & \scriptstyle \left(\frac{ -1,\, \delta}{2_{I}}\right) &
\scriptstyle \left(\frac{\varepsilon_{p},\, \delta}{2_{I}}\right) \\
\scriptstyle \left(\frac{2(x+y\sqrt{p}),\, \delta}{\sqrt{p}}\right)&
\scriptstyle \left(\frac{ 2(x-y\sqrt{p}),\, \delta}{\sqrt{p}}\right)& \scriptstyle \left(\frac{ 2,\, \delta}{\sqrt{p}}\right)& \scriptstyle \left(\frac{ \sqrt{p},\, \delta}{\sqrt{p}}\right)&
\scriptstyle \left(\frac{ -1,\, \delta}{\sqrt{p}}\right)& \scriptstyle \left(\frac{ \varepsilon_{p},\, \delta}{\sqrt{p}}\right)\\ 	
\scriptstyle \left(\frac{2(x+y\sqrt{p}),\, \delta}{p_{\infty}}\right)& \scriptstyle \left(\frac{2(x-y\sqrt{p}),\, \delta}{p_{\infty}}\right)& \scriptstyle \left(\frac{ 2,\, \delta}{p_{\infty}}\right)& \scriptstyle \left(\frac{ \sqrt{p},\, \delta}{p_{\infty}}\right)& \scriptstyle \left(\frac{ -1,\, \delta}{p_{\infty}}\right)& \scriptstyle \left(\frac{\varepsilon_{p},\, \delta}{p_{\infty}}\right) \\	

\end{pmatrix}.$$
We consider the above matrix with coefficients in $\mathbb{F}_{2}$ by replacing  $1$ by $0$ and  $-1$ by $1$.
Consequently,  we can fill the R\'edei matrix according to the values of the prime number $q$.
Suppose that $q\equiv 1 \pmod 8$, then $\left(\frac{2}{q}\right)=1.$
\begin{center}
	
	\begin{tabular}{c||r}
		\hline\hline
		$\bullet$ If $\bf {\left(\frac{p}{q}\right)_{4}=\left(\frac{ q}{p}\right)_{4}=1}$ &$\bullet$ If $\bf {\left(\frac{p}{q}\right)_{4}=\left(\frac{ q}{p}\right)_{4}=-1}$ \\
		\hline
		$ R_{K/k} =\begin{pmatrix}
		\scriptstyle 0 & \scriptstyle 0 & \scriptstyle 0  & \scriptstyle 0 & \scriptstyle 0 & \scriptstyle 0 \\
		\scriptstyle 0 & \scriptstyle 0 & \scriptstyle 0 & \scriptstyle 0 & \scriptstyle 0 & \scriptstyle 0  \\
		\scriptstyle 0 & \scriptstyle 0 & \scriptstyle 1 & \scriptstyle 1 & \scriptstyle 1 & \scriptstyle 0  \\
		\scriptstyle 0 & \scriptstyle 0 & \scriptstyle 1 & \scriptstyle 0 & \scriptstyle 0 & \scriptstyle 1  \\
		\scriptstyle 0 & \scriptstyle 0 &\scriptstyle 0 & \scriptstyle 1 & \scriptstyle 1 & \scriptstyle 1
		\end{pmatrix}, $ & $R_{K/k} = \begin{pmatrix}
		\scriptstyle 0 & \scriptstyle 1 & \scriptstyle 0  & \scriptstyle 1 & \scriptstyle 0 & \scriptstyle 0 \\
		\scriptstyle 1 & \scriptstyle 0 & \scriptstyle 0 & \scriptstyle 1 & \scriptstyle 0 & \scriptstyle 0  \\
		\scriptstyle 0 & \scriptstyle 0 & \scriptstyle 1 & \scriptstyle 1 & \scriptstyle 1 & \scriptstyle 0  \\
		\scriptstyle 1 & \scriptstyle 1 & \scriptstyle 1 & \scriptstyle 0 & \scriptstyle 0 & \scriptstyle 1  \\
		\scriptstyle 0 & \scriptstyle 0 &\scriptstyle 0 & \scriptstyle 1 & \scriptstyle 1 & \scriptstyle 1
		\end{pmatrix}, $  \\
		
		\hline\hline
	\end{tabular}
	\label{tab:hresult}
\end{center}

\begin{center}
	
	\begin{tabular}{c||r}
		\hline\hline
		$\bullet$ If $\bf {\left(\frac{ p}{q}\right)_{4}=-\left(\frac{ q}{p}\right)_{4}=1}$ & $\bullet$ If $\bf {\left(\frac{ p}{q}\right)_{4}=-\left(\frac{ q}{p}\right)_{4}=-1}$ \\
		\hline
		 $ R_{K/k} =\begin{pmatrix}
		 \scriptstyle 0 & \scriptstyle 0 & \scriptstyle 0  & \scriptstyle 0 & \scriptstyle 0 & \scriptstyle 1 \\
		 \scriptstyle 0 & \scriptstyle 0 & \scriptstyle 0 & \scriptstyle 0 & \scriptstyle 0 & \scriptstyle 1  \\
		 \scriptstyle 1 & \scriptstyle 1 & \scriptstyle 1 & \scriptstyle 1 & \scriptstyle 1 & \scriptstyle 0  \\
		 \scriptstyle 1 & \scriptstyle 1 & \scriptstyle 1 & \scriptstyle 0 & \scriptstyle 0 & \scriptstyle 1  \\
		 \scriptstyle 0 & \scriptstyle 0 &\scriptstyle 0 & \scriptstyle 1 & \scriptstyle 1 & \scriptstyle 1
		 \end{pmatrix}, $
		 &$R_{K/k} =\begin{pmatrix}
		 \scriptstyle 0 & \scriptstyle 1 & \scriptstyle 0  & \scriptstyle 1 & \scriptstyle 0 & \scriptstyle 1 \\
		 \scriptstyle 1 & \scriptstyle 0 & \scriptstyle 0 & \scriptstyle 1 & \scriptstyle 0 & \scriptstyle 1  \\
		 \scriptstyle 1 & \scriptstyle 1 & \scriptstyle 1 & \scriptstyle 1 & \scriptstyle 1 & \scriptstyle 0  \\
		 \scriptstyle 0 & \scriptstyle 0 & \scriptstyle 1 & \scriptstyle 0 & \scriptstyle 0 & \scriptstyle 1  \\
		 \scriptstyle 0 & \scriptstyle 0 &\scriptstyle 0 & \scriptstyle 1 & \scriptstyle 1 & \scriptstyle 1
		 \end{pmatrix}. $  \\
		
		\hline\hline
	\end{tabular}
	\label{tab:hresult}
\end{center}

Consequently,  $\mathrm{rank}(R_{K/k})=\begin{cases} 4, & \text{if $\left(\frac{ p}{q}\right)_{4}=-1$, } \\
3, & \text{if $\left(\frac{ p}{q}\right)_{4}=-\left(\frac{ q}{p}\right)_{4}=1$, } \\

2, & \text{if $\left(\frac{ p}{q}\right)_{4}=\left(\frac{ q}{p}\right)_{4}=1$. }
\end{cases}$
	
Thus from the $4$-rank formula of $Cl(K)$ (see, \eqref{4RF}),  we obtain 
$$ r_{4}(K)=\begin{cases} 0, & \text{if $\left(\frac{ p}{q}\right)_{4}=-1$, } \\
1, & \text{if $\left(\frac{ p}{q}\right)_{4}=-\left(\frac{ q}{p}\right)_{4}=1$, } \\
2, & \text{if $\left(\frac{ p}{q}\right)_{4}=\left(\frac{ q}{p}\right)_{4}=1$. }
\end{cases}$$

Suppose that $q\equiv 5 \pmod 8$, then $\left(\frac{2}{q}\right)=-1. $
\begin{center}
\begin{tabular}{c||r}
	\hline\hline
$\bullet$ If $\bf {\left(\frac{p}{q}\right)_{4}=\left(\frac{ q}{p}\right)_{4}=1}$ & $\bullet$ If $\bf {\left(\frac{p}{q}\right)_{4}=\left(\frac{ q}{p}\right)_{4}=-1}$\\
	\hline
	$ R_{K/k} =\begin{pmatrix}
	\scriptstyle 0 & \scriptstyle 0 & \scriptstyle 1  & \scriptstyle 0 & \scriptstyle 0 & \scriptstyle 0 \\
	\scriptstyle 0 & \scriptstyle 0 & \scriptstyle 1 & \scriptstyle 0 & \scriptstyle 0 & \scriptstyle 0  \\
	\scriptstyle 0 & \scriptstyle 0 & \scriptstyle 1 & \scriptstyle 1 & \scriptstyle 1 & \scriptstyle 0  \\
	\scriptstyle 0 & \scriptstyle 0 & \scriptstyle 1 & \scriptstyle 0 & \scriptstyle 0 & \scriptstyle 1  \\
	\scriptstyle 0 & \scriptstyle 0 &\scriptstyle 0 & \scriptstyle 1 & \scriptstyle 1 & \scriptstyle 1  \\
	\end{pmatrix}, $
	& $R_{K/k} = \begin{pmatrix}
	\scriptstyle 0 & \scriptstyle 1 & \scriptstyle 1  & \scriptstyle 1 & \scriptstyle 0 & \scriptstyle 0 \\
	\scriptstyle 1 & \scriptstyle 0 & \scriptstyle 1 & \scriptstyle 1 & \scriptstyle 0 & \scriptstyle 0  \\
	\scriptstyle 0 & \scriptstyle 0 & \scriptstyle 1 & \scriptstyle 1 & \scriptstyle 1 & \scriptstyle 0  \\
	\scriptstyle 1 & \scriptstyle 1 & \scriptstyle 1 & \scriptstyle 0 & \scriptstyle 0 & \scriptstyle 1  \\
	\scriptstyle 0 & \scriptstyle 0 &\scriptstyle 0 & \scriptstyle 1 & \scriptstyle 1 & \scriptstyle 1  \\
	
	\end{pmatrix}, $ \\
	
	\hline\hline
\end{tabular}
\end{center}

\begin{center}
	\begin{tabular}{c||r}
		\hline\hline
		$\bullet$ If $\bf {\left(\frac{ p}{q}\right)_{4}=-\left(\frac{ q}{p}\right)_{4}=1}$ & $\bullet$ If $\bf {\left(\frac{ p}{q}\right)_{4}=-\left(\frac{ q}{p}\right)_{4}=-1}$ \\
		\hline
		$ R_{K/k} =\begin{pmatrix}
		\scriptstyle 0 & \scriptstyle 0 & \scriptstyle 1  & \scriptstyle 0 & \scriptstyle 0 & \scriptstyle 1 \\
		\scriptstyle 0 & \scriptstyle 0 & \scriptstyle 1 & \scriptstyle 0 & \scriptstyle 0 & \scriptstyle 1  \\
		\scriptstyle 1 & \scriptstyle 1 & \scriptstyle 1 & \scriptstyle 1 & \scriptstyle 1 & \scriptstyle 0  \\
		\scriptstyle 1 & \scriptstyle 1 & \scriptstyle 1 & \scriptstyle 0 & \scriptstyle 0 & \scriptstyle 1  \\
		\scriptstyle 0 & \scriptstyle 0 &\scriptstyle 0 & \scriptstyle 1 & \scriptstyle 1 & \scriptstyle 1  \\
		
		\end{pmatrix}, $
		&$R_{K/k} =\begin{pmatrix}
		\scriptstyle 0 & \scriptstyle 1 & \scriptstyle 1 & \scriptstyle 1 & \scriptstyle 0 & \scriptstyle 1 \\
		\scriptstyle 1 & \scriptstyle 0 & \scriptstyle 1 & \scriptstyle 1 & \scriptstyle 0 & \scriptstyle 1  \\
		\scriptstyle 1 & \scriptstyle 1 & \scriptstyle 1 & \scriptstyle 1 & \scriptstyle 1 & \scriptstyle 0  \\
		\scriptstyle 0 & \scriptstyle 0 & \scriptstyle 1 & \scriptstyle 0 & \scriptstyle 0 & \scriptstyle 1  \\
		\scriptstyle 0 & \scriptstyle 0 &\scriptstyle  0 & \scriptstyle 1 & \scriptstyle 1 & \scriptstyle 1  \\
		
		\end{pmatrix}. $  \\
		
		\hline\hline
	\end{tabular}
\end{center}
Consequently,  $\mathrm{rank}(R_{K/k})=\begin{cases} 3, & \text{if $\left(\frac{ p}{q}\right)_{4}=1$, } \\
4, & \text{if $\left(\frac{ p}{q}\right)_{4}=-1$. }
\end{cases}$
	
Then $r_{4}(K)=\begin{cases} 1, & \text{if $\left(\frac{ p}{q}\right)_{4}=1$, } \\
0, & \text{if $\left(\frac{ p}{q}\right)_{4}=-1$. }
\end{cases}$

Suppose that $q\equiv 3 \pmod 8$, then $\left(\frac{2}{q}\right)=-1. $\\
If $\left(\frac{y}{q}\right)=1$
\begin{center}
	\begin{tabular}{c||r}
		\hline\hline
		$\bullet$ If $  \bf { \left(\frac{ q}{p}\right)_{4}=1}$ & $\bullet$  If $\bf {\left(\frac{ q}{p}\right)_{4}=-1}$ \\
		\hline
		$ R_{K/k} = \begin{pmatrix}
		
		\scriptstyle 0 & \scriptstyle 1 & \scriptstyle 1 & \scriptstyle 0 & \scriptstyle 1 & \scriptstyle 1 \\
		\scriptstyle 1 & \scriptstyle 1 & \scriptstyle 1 & \scriptstyle 1 & \scriptstyle 1 & \scriptstyle 0\\
		\scriptstyle 1 & \scriptstyle 0 & \scriptstyle 1 & \scriptstyle 0 & \scriptstyle 1 & \scriptstyle 1 \\
		\scriptstyle 0 & \scriptstyle 0 & \scriptstyle 1 & \scriptstyle 0 & \scriptstyle 0 & \scriptstyle 1 \\
		\scriptstyle 0 & \scriptstyle 0 &\scriptstyle  0 & \scriptstyle 1 & \scriptstyle 1 & \scriptstyle 1 \\
		
		\end{pmatrix}, $
		& $R_{K/k} = \begin{pmatrix}
		
		\scriptstyle 0 & \scriptstyle 0 & \scriptstyle 1 & \scriptstyle 1 & \scriptstyle 1 & \scriptstyle 1\\
		\scriptstyle 0 & \scriptstyle 1 & \scriptstyle 1 & \scriptstyle 0 & \scriptstyle 1 & \scriptstyle 0 \\
		\scriptstyle 1 & \scriptstyle 0 & \scriptstyle 1 & \scriptstyle 0 & \scriptstyle 1 & \scriptstyle 1 \\
		\scriptstyle 1 & \scriptstyle 1 & \scriptstyle 1 & \scriptstyle 0 & \scriptstyle 0 & \scriptstyle 1 \\
		\scriptstyle 0 & \scriptstyle 0 & \scriptstyle 0 & \scriptstyle 1 & \scriptstyle 1 & \scriptstyle 1 \\
		
		\end{pmatrix}.$ \\
		
		\hline\hline
	\end{tabular}
\end{center}

If $\left(\frac{y}{q}\right)=-1$
\begin{center}
	\begin{tabular}{c||r}
		\hline\hline
		$\bullet$ If $  \bf { \left(\frac{ q}{p}\right)_{4}=1}$ & $\bf {\left(\frac{ q}{p}\right)_{4}=-1}$ \\
		\hline
		$ R_{K/k} = \begin{pmatrix}
		
		\scriptstyle 1 & \scriptstyle 1 & \scriptstyle 1 & \scriptstyle 1 & \scriptstyle 1 & \scriptstyle 0 \\
		\scriptstyle 1 & \scriptstyle 0 & \scriptstyle 1 & \scriptstyle 0 & \scriptstyle 1 & \scriptstyle 1 \\
		\scriptstyle 0 & \scriptstyle 1 & \scriptstyle 1 & \scriptstyle 0 & \scriptstyle 1 & \scriptstyle 1 \\
		\scriptstyle 0 & \scriptstyle 0 & \scriptstyle 1 & \scriptstyle 0 & \scriptstyle 0 & \scriptstyle 1 \\
		\scriptstyle 0 & \scriptstyle 0 &\scriptstyle  0 & \scriptstyle 1 & \scriptstyle 1 & \scriptstyle 1
		\end{pmatrix}, $
		& $R_{K/k} = \begin{pmatrix}
		
		\scriptstyle 1 & \scriptstyle 0 & \scriptstyle 1 & \scriptstyle 0 & \scriptstyle 1 & \scriptstyle 0\\
		\scriptstyle 0 & \scriptstyle 0 & \scriptstyle 1 & \scriptstyle 1 & \scriptstyle 1 & \scriptstyle 1 \\
		\scriptstyle 0 & \scriptstyle 1 & \scriptstyle 1 & \scriptstyle 0 & \scriptstyle 1 & \scriptstyle 1 \\
		\scriptstyle 1 & \scriptstyle 1 & \scriptstyle 1 & \scriptstyle 0 & \scriptstyle 0 & \scriptstyle 1 \\
		\scriptstyle 0 & \scriptstyle 0 & \scriptstyle 0 & \scriptstyle 1 & \scriptstyle 1 & \scriptstyle 1
		\end{pmatrix}. $  \\
		
		\hline\hline
	\end{tabular}
\end{center}
		
So $ \mathrm{rank}(R_{K/k})=4$  and thus  $ r_{4}(K)=0$.\\
Suppose that $q\equiv 7 \pmod 8$, then $\left(\frac{2}{q}\right)=1. $\\
If $\left(\frac{y}{q}\right)=1$
\begin{center}
	\begin{tabular}{c||r}
		\hline\hline
		$\bullet$ If $  \bf { \left(\frac{ q}{p}\right)_{4}=1}$ & $\bullet$  If $\bf {\left(\frac{ q}{p}\right)_{4}=-1}$ \\
		\hline
		$ R_{K/k} = \begin{pmatrix}
		
		\scriptstyle 0 & \scriptstyle 1 & \scriptstyle 0 & \scriptstyle 0 & \scriptstyle 1 & \scriptstyle 1 \\
		\scriptstyle 1 & \scriptstyle 1 & \scriptstyle 0 & \scriptstyle 1 & \scriptstyle 1 & \scriptstyle 0 \\
		\scriptstyle 1 & \scriptstyle 0 & \scriptstyle 1 & \scriptstyle 0 & \scriptstyle 1 & \scriptstyle 1  \\
		\scriptstyle 0 & \scriptstyle 0 & \scriptstyle 1 & \scriptstyle 0 & \scriptstyle 0 & \scriptstyle 1  \\
		\scriptstyle 0 & \scriptstyle 0 & \scriptstyle 0 & \scriptstyle 1 & \scriptstyle 1 & \scriptstyle 1  \\
		
		\end{pmatrix}, $
		& $R_{K/k} = \begin{pmatrix}
		
		\scriptstyle 0 & \scriptstyle 0 & \scriptstyle 0 & \scriptstyle 1 & \scriptstyle 1 & \scriptstyle 1 \\
		\scriptstyle 0 & \scriptstyle 1 & \scriptstyle 0 & \scriptstyle 0 & \scriptstyle 1 & \scriptstyle 0 \\
		\scriptstyle 1 & \scriptstyle 0 & \scriptstyle 1 & \scriptstyle 0 & \scriptstyle 1 & \scriptstyle 1  \\
		\scriptstyle 1 & \scriptstyle 1 & \scriptstyle 1 & \scriptstyle 0 & \scriptstyle 0 & \scriptstyle 1  \\
		\scriptstyle 0 & \scriptstyle 0 & \scriptstyle 0 & \scriptstyle 1 & \scriptstyle 1 & \scriptstyle 1  \\
		
		\end{pmatrix}. $ \\
		
		\hline\hline
	\end{tabular}
\end{center}

If $\left(\frac{y}{q}\right)=-1$
\begin{center}
	\begin{tabular}{c||r}
		\hline\hline
		$\bullet$ If $  \bf { \left(\frac{ q}{p}\right)_{4}=1}$ & $\bullet$  If $\bf {\left(\frac{ q}{p}\right)_{4}=-1}$ \\
		\hline
		$ R_{K/k} = \begin{pmatrix}
		
		\scriptstyle 1 & \scriptstyle 1 & \scriptstyle 0 & \scriptstyle 1 & \scriptstyle 1 & \scriptstyle 0 \\
		\scriptstyle 1 & \scriptstyle 0 & \scriptstyle 0 & \scriptstyle 0 & \scriptstyle 1 & \scriptstyle 1 \\
		\scriptstyle 0 & \scriptstyle 1 & \scriptstyle 1 & \scriptstyle 0 & \scriptstyle 1 & \scriptstyle 1  \\
		\scriptstyle 0 & \scriptstyle 0 & \scriptstyle 1 & \scriptstyle 0 & \scriptstyle 0 & \scriptstyle 1  \\
		\scriptstyle 0 & \scriptstyle 0 & \scriptstyle 0 & \scriptstyle 1 & \scriptstyle 1 & \scriptstyle 1  \\
		
		\end{pmatrix}, $
		& $R_{K/k} = \begin{pmatrix}
		
		\scriptstyle 1 & \scriptstyle 0 & \scriptstyle 0 & \scriptstyle 0 & \scriptstyle 1 & \scriptstyle 0 \\
		\scriptstyle 0 & \scriptstyle 0 & \scriptstyle 0 & \scriptstyle 1 & \scriptstyle 1 & \scriptstyle 1 \\
		\scriptstyle 0 & \scriptstyle 1 & \scriptstyle 1 & \scriptstyle 0 & \scriptstyle 1 & \scriptstyle 1  \\
		\scriptstyle 1 & \scriptstyle 1 & \scriptstyle 1 & \scriptstyle 0 & \scriptstyle 0 & \scriptstyle 1  \\
		\scriptstyle 0 & \scriptstyle 0 & \scriptstyle 0 & \scriptstyle 1 & \scriptstyle 1 & \scriptstyle 1  \\
		
		\end{pmatrix}. $ \\
		
		\hline\hline
	\end{tabular}
\end{center}

Consequently, $ \mathrm{rank}(R_{K/k})=\begin{cases} 4, & \text{if $\left(\frac{ q}{p}\right)_{4}=1$, } \\
3, & \text{if $\left(\frac{ q}{p}\right)_{4}=-1$. }
\end{cases}$

Then
$r_{4}(K)=\begin{cases} 0, & \text{if $\left(\frac{ q}{p}\right)_{4}=1$, } \\
1, & \text{if $\left(\frac{ q}{p}\right)_{4}=-1$. }
\end{cases}$
  \end{proof}

  \begin{example}
  	Keep  previous hypotheses and notations. Here we give examples when $q \equiv 1 \pmod 8$.
  	\begin{itemize}
  		\item [$ (i) $] For $p=173$ and $q=41$, we have $173\equiv5\pmod 8$ and $\left(\frac{173}{41}\right)=1$. From Corollary~((\ref{SHC})-(a)), we obtain $r_{2}(K)=2.$ Also $41\equiv 1 \pmod 8$ and $\left(\frac{173 }{41}\right)_{4}=-1 $, so the condition of  Theorem~( (\ref{Thb})-(1)) are satisfied. Hence,  $r_{4}(K) =0$, i.e., $ Cl_2(K)\cong \mathbb{Z}/2\mathbb{Z}\times \mathbb{Z}/2\mathbb{Z} $.
  		\item [$ (ii) $] For $p=101$ and $q=17$, we have $101\equiv5\pmod 8$ and $ \left(\frac{101}{17}\right)=1 $.  From Corollary((\ref{SHC})-(a)), we obtain $r_{2}(K) = 2$.
  		As $17\equiv 1 \pmod 8$ and $ \left(\frac{101 }{17}\right)_{4}=-\left(\frac{17}{101}\right)_{4}=1 $, so the condition of  Theorem~((\ref{Thb})-(1)) are satisfied. Hence, $r_{4}(K)=1$.
  		In fact, by Pari/GP we get $Cl_2(K)\cong \mathbb{Z}/8\mathbb{Z}\times \mathbb{Z}/2\mathbb{Z} $.
  		\item [$ (iii) $] For $p=157$ and $q=17$, we have $157\equiv5\pmod 8$ and $ \left(\frac{157}{17}\right)=1 $.  From Corollary((\ref{SHC})-(a)), we obtain $r_{2}(K)=2$. Also $17\equiv 1 \pmod 8$ and $\left(\frac{157 }{17}\right)_{4}=\left(\frac{17}{157}\right)_{4}=1$, so the condition of Theorem~((\ref{Thb})-(1)) are satisfied. Hence, $r_{4}(K) = 2$. Precisely, $ Cl_2(K)\cong \mathbb{Z}/8\mathbb{Z}\times \mathbb{Z}/4\mathbb{Z}$.
  	\end{itemize}	
  	
  \end{example}
  \begin{example}
  Now an example when $q\equiv 3\pmod 8$.
  	\begin{itemize}
  	\item [ ]For $ p =53 $ and $ q =43 $, we have $53\equiv5\pmod 8$ and $ \left(\frac{53}{43}\right)=1 $.  From Corollary((\ref{SHC})-(a)), we obtain $r_{2}(K) =2$, and as $43\equiv 3\pmod 8$, with the condition of Theorem~((\ref{Thb})-(2)), hence  $r_{4}(K) =0$. In other words, $ Cl_2(K)\cong \mathbb{Z}/2\mathbb{Z}\times \mathbb{Z}/2\mathbb{Z}$.
  	\end{itemize}
  \end{example}
  \begin{example}
  	The following two examples illustrate the 3rd case of Theorem \ref{Thb}.
  	\begin{itemize}
  		\item [$ (i) $] For $ p =269 $ and $ q =53 $, we have $269\equiv5\pmod 8$ and $ \left(\frac{269}{53}\right)=1 $.  The Corollary((\ref{SHC})-(a))  implies that, $r_{2}(K) =2$ and since the condition of Theorem~((\ref{Thb})-(3)) are satisfied, so,  $r_{4}(K) =0$, i.e., $ Cl_2(K)\cong \mathbb{Z}/2\mathbb{Z}\times \mathbb{Z}/2\mathbb{Z} $.
  		\item [$ (ii) $] For $p = 293$ and $q=109$, we can easily verify that $293\equiv5\pmod 8$ and $ \left(\frac{293}{109}\right)=1 $, $109\equiv 5 \pmod 8$ and  $\left(\frac{293 }{109}\right)_{4}=1$, then the Theorem~((\ref{Thb})-(3)) implies that $r_{4}(K) =1$,	exactly, $ Cl_2(K)\cong \mathbb{Z}/4\mathbb{Z}\times \mathbb{Z}/2\mathbb{Z} $.
  	\end{itemize}
  	
  \end{example}
  \begin{example}
  	We end with an example for the last case of  Theorem \ref{Thb}.
  	\begin{itemize}
  		\item [$ (i) $] For $ p =173 $ and $ q =23 $, we have $173\equiv5\pmod 8$ and $ \left(\frac{173}{23}\right)=1 $,  $23\equiv 7\pmod 8$ and  $ \left(\frac{23 }{173}\right)_{4}=1 $. According to  Corollary((\ref{SHC})-(a)) and Theorem~((\ref{Thb})-(4)), we have  $r_{4}(K) =0$.
  		\item [$ (ii) $] Similarly for $ p =149 $ and $ q =7 $, we have $149\equiv5\pmod 8$, $ \left(\frac{149}{7}\right)=1 $,  $7\equiv 7 \pmod 8$ and $\left(\frac{7 }{149}\right)_{4}=-1$, we find that $r_{4}(K) =1$. By Pari/GP, we find  $ Cl_2(K)\cong \mathbb{Z}/16\mathbb{Z}\times \mathbb{Z}/2\mathbb{Z} $.
  	
  	\end{itemize}
  	
  \end{example}
\subsection{Case:  $d=q_{1}q_{2} $ with $ q_{1}$, $q_{2} $ are distinct prime numbers and  $\left(\frac{p }{q_{1}}\right)=\left(\frac{p }{q_{2}}\right)=-1$}
Let $ p$,  $ q_{1}$,  $q_{2} $ be different odd prime numbers  such that $p\equiv 5 \pmod 8$,  $\left(\frac{p }{q_{1}}\right)=\left(\frac{p }{q_{2}}\right)=-1$.
Then the finite primes of $ k$ ramifying in $K$, in this case, are:  $\tilde{q_{1}} $, $\tilde{q_{2}} $, $ 2_{I}$ and $(\sqrt{p})$,  where $\tilde{q_{1}}=q_{1}O_{k}$, $\tilde{q_{2}}=q_{2}O_{k}$,   $pO_{k}=(\sqrt{p})^{2}$ and $2_{I}=2O_{k}$.
To compute the $4$-rank of the class group of $K$, by using the generalized R\'edei-matrix,  we need the following lemma:
\begin{lemma}
	Keeping previous hypotheses and notations, then	
	\begin{enumerate}[\rm a.]
		\item For $i=1, 2$ we have
		$\left(\frac{ \varepsilon_{p}, \, \delta }{\tilde{q_{i}}}\right)=\left(\frac{ -1 }{q_{i}}\right)$
			and $ \left(\frac{2, \, \delta }{\tilde{q_{i}}}\right)=\left(\frac{-1, \, \delta }{\tilde{q_{i}}}\right)=1$.	
		\item  $\left(\frac{q_{i}, \, \delta }{\tilde{q_{i}}}\right)=-\left(\frac{ -1 }{q_{i}}\right)$ and $ \left(\frac{q_{i}, \, \delta }{\tilde{q_{j}}}\right)=1$ for $ i\neq j\in\{ 1, 2\}$.
		
	\end{enumerate}		
\end{lemma}

\begin{proof}We have
\begin{enumerate}[\rm a.]
		\item
		\begin{itemize}
			\item [$ \bullet $] $  \left(\frac{-1, \, \delta }{\tilde{q_{1}}}\right)=\left[\frac{ -1 }{\tilde{q_{1}}}\right]=\left(\frac{ 1 }{q_{1}}\right)=1$.	
			\item [$ \bullet $] As $N_{k/\mathbb{Q}}(\varepsilon_{p}) = -1$, so by (\cite[{Ch 4 \textsection} 4.1 Proposition 4.2, p. 112]{L00}), \\
			$\left(\frac{ \varepsilon_{p}, \, \delta }{\tilde{q_{1}}}\right)= \left[\frac{  \varepsilon_{p} }{\tilde{q_{1}}}\right]
			=\left(\frac{ N_{k/\mathbb{Q}}(\varepsilon_{p}) }{q_{1}}\right)=\left(\frac{ -1}{q_{1}}\right)$,
			and $\left(\frac{ \varepsilon_{p}, \, \delta }{\tilde{q_{2}}}\right)=\left(\frac{ -1}{q_{2}}\right)$.\\
			By the same argument, we get $ \left(\frac{-1, \, \delta }{\tilde{q_{2}} }\right)=\left[\frac{ -1 }{\tilde{q_{2}}}\right]=\left(\frac{1 }{q_{2}}\right) =1$.
			
			\item [$ \bullet $]	 As $N_{k/\mathbb{Q}}(2) =4$, so by  (\cite[{Ch 4 \textsection} 4.1 Proposition 4.2, p. 112]{L00}), $ \left(\frac{2, \, \delta }{\tilde{q_{1}}}\right)=\left[\frac{ 2 }{\tilde{q_{1}}}\right]=1$.
			Similarly,  we have
			$ \left(\frac{2, \, \delta }{\tilde{q_{2}}}\right)=1$.	
		\end{itemize}		
		
		\item
		$
		\left(\frac{q_{1}, \, \delta }{\tilde{q_{2}}}\right)= \left[\frac{ q_{1}}{\tilde{q_{2}}}\right]=\left(\frac{ q_{1}^{2}}{q_{2}}\right)=1,
		$
		and\\
		$ \begin{array}{rcl}
		\left(\frac{q_{1}, \, \delta }{\tilde{q_{1}}}\right)
		&=&\left(\frac{q_{1},\, -q_{1} }{\tilde{q_{1}}}\right)\left(\frac{q_{1},\, -q_{2}\sqrt{p}  }{\tilde{q_{1}}}\right)
		=\left(\frac{q_{1},\, -q_{2}\sqrt{p}  }{\tilde{q_{1}}}\right)
		=\left[\frac{  -q_{2}\sqrt{p}}{\tilde{q_{1}}}\right]\\
		&=&\left[\frac{  -\sqrt{p}}{\tilde{q_{1}}}\right]
		=\left(\frac{-1}{q_{1}}\right)\left(\frac{p}{q_{1}}\right)
		=-\left(\frac{-1}{q_{1}}\right).
		\end{array}	$	\\
		By the same argument, we get
		$\left(\frac{q_{2}, \, \delta }{\tilde{q_{1}}}\right)=1$	and $ \left(\frac{q_{2}, \, \delta }{\tilde{q_{2}}}\right)=-\left(\frac{ -1 }{q_{2}}\right)$.
\end{enumerate}	
	
\end{proof}

\begin{remark}
 By similar calculations as in Lemma $(\ref{SHL2})$,  we get
  \begin{center}$ \left(\frac{q_{1}, \, \delta }{\sqrt{p}}\right)=\left(\frac{q_{2}, \, \delta }{\sqrt{p}}\right)=-1  $ and $ \left(\frac{q_{1}, \, \delta }{p_{\infty}}\right)=\left(\frac{q_{2}, \, \delta }{p_{\infty}}\right)=1.$\end{center}
		 The symbols $\left(\frac{-, \, \delta }{2_{I}}\right) $ can be computed by using product formula  except for  $ \left(\frac{\sqrt{p}, \, \delta }{\mathcal{P}_{i}}\right) $ which can be calculated by the formula $\left(\frac{\sqrt{p}, \, \delta }{\mathcal{P}_{i}}\right)=\left(\frac{-1, \, \delta }{\mathcal{P}_{i}}\right)\left(\frac{q_{1}, \, \delta }{\mathcal{P}_{i}}\right)\left(\frac{q_{2}, \, \delta }{\mathcal{P}_{i}}\right) $ consequence of $ \left(\frac{-\delta, \, \delta }{\mathcal{P}_{i}}\right)=1 $ for $ \mathcal{P}_{i}=\tilde{q_{1}},  \tilde{q_{2}} $, $ 2_{I}, (\sqrt{p})\text{ or }p_{\infty}$.\\
If $ d=q_{1}q_{2} $,  $\left(\frac{p }{q_{1}}\right)=\left(\frac{p }{q_{2}}\right)=-1$, then we can summarise the  results in the following table:
	\begin{center}
		\begin{tabular}{||c||c|c|c|c|c||}
			\hline\hline
			
			\diagbox{element}{prime} & $ \tilde{q_{1}} $ & $\tilde{q_{2}} $ & $2_{I} $ & $ (\sqrt{p})$& $ p_{\infty} $\\
			\hline
			$ q_{1}$ & $ -\left(\frac{ -1 }{q_{1}}\right) $  & 1& $ \left(\frac{ -1 }{q_{1}}\right) $ & -1 & 1 \\
			\hline
			$ q_{2}$ & 1 &$ -\left(\frac{ -1 }{q_{2}}\right) $ & $\left(\frac{ -1 }{q_{2}}\right) $ & -1  & 1 \\
			\hline
			$ 2 $ & 1 & 1 &-1 &-1 &1 \\
			\hline
			$\sqrt{p}$  & $-\left(\frac{ -1 }{q_{1}}\right) $ & $-\left(\frac{ -1 }{q_{2}}\right) $ & $ -\left(\frac{ -1 }{d}\right) $ &1&-1 \\			
			\hline
			-1 & 1& 1 & -1 & 1 &  -1\\
			\hline
			$ \varepsilon_{p} $ &  $ \left(\frac{ -1 }{q_{1}}\right) $ & $ \left(\frac{ -1 }{q_{2}}\right) $ & $ \left(\frac{ -1 }{d}\right) $ & -1 & -1 \\
			\hline\hline   			  			  			
		\end{tabular}
	\end{center}	
	
\end{remark}

\begin{proof}[\rm\textbf{Proof of Theorem C}]
	If $ d=q_{1}q_{2} $ with $ q_{1}$,  $q_{2} $ are two different prime numbers such that $\left(\frac{p }{q_{1}}\right)=\left(\frac{p }{q_{2}}\right)=-1$,  by corollary  $ (\ref{SHC})$ we have $ r_{2}(K)=2 $.
	As the finite primes of $ k $ which ramify in $K$ are: 	$\tilde{q_{1}}$, $\tilde{q_{2}}$, $2_{I}$, $(\sqrt{p})$
	where  $\tilde{q_{1}}=q_{1}O_{k}$, $\tilde{q_{2}}=q_{2}O_{k} $,  $2_{I}=2O_{k}$ and $(\sqrt{p})^{2}=pO_{k}$.
	So there exist prime ideals $\mathcal{H}_{1}, \mathcal{H}_{2}, \mathcal{H}_{3}, \mathcal{H}_{4}$ of $K$ such that
	$$\mathcal{H}_{1}^{2}=\tilde{q}_{1}O_{K}, \mathcal{H}_{2}^{2}=\tilde{q}_{2}O_{K}, \mathcal{H}_{3}^{2}=2_{I}O_{K},  \mathcal{H}_{4}^{2}=\sqrt{p}O_{K}.  $$	
	Thus~\
	\begin{itemize}
		\item  $\mathcal{H}_{1}\sigma(\mathcal{H}_{1})=\mathcal{H}_{1}\mathcal{H}_{1}O_{K}=\tilde{q_{1}}O_{K}=q_{1}O_{K}, $
		\item  $\mathcal{H}_{2}\sigma(\mathcal{H}_{2})=\mathcal{H}_{2}\mathcal{H}_{2}O_{K}=\tilde{q_{2}}O_{K}=q_{2}O_{K}, $
		\item  $\mathcal{H}_{3}\sigma(\mathcal{H}_{3})=\mathcal{H}_{3}\mathcal{H}_{3}O_{K}=2_{I}O_{K}=2O_{K}, $
		\item  $\mathcal{H}_{4}\sigma(\mathcal{H}_{4})=\mathcal{H}_{4}\mathcal{H}_{4}O_{K}=\sqrt{p}O_{K}=\sqrt{p}O_{K}. $
	\end{itemize}
	where $\sigma$ is the generateur of the Galois group of $K$ over $k$.\\
	Moreover,  $E_{k}=\langle-1, \varepsilon_{p}\rangle$ and $ -1$, $\varepsilon_{p}$, $-\varepsilon_{p}\notin E_{k}\cap N_{K/k}(K) $ (Lemma \ref{SHL1}).
	Then the $ 5\times(4 + 2)  $ generalized R\'edei's matrix   $ R_{K/k} $ is  	
	$$ R_{K/k} = \begin{pmatrix}
	\scriptstyle \left(\frac{q_{1}, \, \delta }{\tilde{q}_{1}}\right)& \scriptstyle \left(\frac{q_{2},\, \delta }{\tilde{q}_{1}}\right)& \scriptstyle \left(\frac{ 2,\, \delta }{\tilde{q}_{1}}\right)& \scriptstyle \left(\frac{ \sqrt{p}, \, \delta }{\tilde{q}_{1}}\right)&
	\scriptstyle \left(\frac{ -1, \, \delta }{\tilde{q}_{1}}\right)&
	\scriptstyle \left(\frac{ \varepsilon_{p}, \, \delta }{\tilde{q}_{1}}\right) \\
	\scriptstyle \left(\frac{q_{1}, \, \delta }{\tilde{q}_{2}}\right)& \scriptstyle \left(\frac{q_{2}, \, \delta }{\tilde{q}_{2}}\right)& \scriptstyle \left(\frac{ 2, \, \delta }{\tilde{q}_{2}}\right)& \scriptstyle \left(\frac{ \sqrt{p}, \, \delta }{\tilde{q}_{2}}\right)& \scriptstyle \left(\frac{ -1, \, \delta }{\tilde{q}_{2}}\right)& \scriptstyle \left(\frac{\varepsilon_{p}, \, \delta }{\tilde{q}_{2}}\right)  \\
	\scriptstyle \left(\frac{q_{1}, \, \delta }{2_{I}}\right)&
	\scriptstyle \left(\frac{ q_{2}, \, \delta }{2_{I}}\right) &
	\scriptstyle \left(\frac{ 2, \, \delta }{2_{I}}\right)&
	\scriptstyle \left(\frac{ \sqrt{p}, \, \delta }{2_{I}}\right) & \scriptstyle \left(\frac{ -1, \, \delta }{2_{I}}\right) &
	\scriptstyle \left(\frac{\varepsilon_{p}, \, \delta }{2_{I}}\right) \\
	\scriptstyle \left(\frac{q_{1}, \, \delta }{\sqrt{p}}\right)&
	\scriptstyle \left(\frac{ q_{2}, \, \delta }{\sqrt{p}}\right)& \scriptstyle \left(\frac{ 2, \, \delta }{\sqrt{p}}\right)& \scriptstyle \left(\frac{ \sqrt{p}, \, \delta }{\sqrt{p}}\right)&
	\scriptstyle \left(\frac{ -1, \, \delta }{\sqrt{p}}\right)& \scriptstyle \left(\frac{ \varepsilon_{p},\, \delta }{\sqrt{p}}\right)\\ 	
	\scriptstyle \left(\frac{ q_{1}, \, \delta }{p_{\infty}}\right)& \scriptstyle \left(\frac{q_{2}, \, \delta }{p_{\infty}}\right)& \scriptstyle \left(\frac{ 2, \, \delta }{p_{\infty}}\right)& \scriptstyle \left(\frac{ \sqrt{p}, \, \delta }{p_{\infty}}\right)& \scriptstyle \left(\frac{ -1, \, \delta }{p_{\infty}}\right)& \scriptstyle \left(\frac{\varepsilon_{p}, \, \delta }{p_{\infty}}\right) \\	
	
	\end{pmatrix}.  $$
	We consider the above matrix with coefficients in $ \mathbb{F}_{2} $ by replacing  $ 1 $ by $ 0 $ and  $ -1 $ by $ 1 $.
	\begin{center}
		\begin{tabular}{c||r}
			\hline\hline
			$\bullet$ If $(q_{1}, q_{2})\equiv (1, 1) \pmod 4$ & $\bullet$  If $(q_{1}, q_{2})\equiv (1, 3) \pmod 4$ \\
			\hline
			$ R_{K/k} = \begin{pmatrix}
			\scriptstyle 1 & \scriptstyle 0 & \scriptstyle 0 & \scriptstyle 1 & \scriptstyle 0 & \scriptstyle 0 \\
			\scriptstyle 0 & \scriptstyle 1 & \scriptstyle 0 & \scriptstyle 1 & \scriptstyle 0 & \scriptstyle 0  \\
			\scriptstyle 0 & \scriptstyle 0 & \scriptstyle 1 & \scriptstyle 1 & \scriptstyle 1 & \scriptstyle 0  \\
			\scriptstyle 1 & \scriptstyle 1 & \scriptstyle 1 & \scriptstyle 0 & \scriptstyle 0 & \scriptstyle 1  \\
			\scriptstyle 0 & \scriptstyle 0 &\scriptstyle 0 & \scriptstyle 1 & \scriptstyle 1 & \scriptstyle 1  \\
			
			\end{pmatrix} $
			& $R_{K/k} = \begin{pmatrix}
			\scriptstyle 1 & \scriptstyle 0 & \scriptstyle 0  & \scriptstyle 1 & \scriptstyle 0 & \scriptstyle 0 \\
			\scriptstyle 0 & \scriptstyle 0 & \scriptstyle 0 & \scriptstyle 0 & \scriptstyle 0 & \scriptstyle 1  \\
			\scriptstyle 0 & \scriptstyle 1 & \scriptstyle 1 & \scriptstyle 0 & \scriptstyle 1 & \scriptstyle 1  \\
			\scriptstyle 1 & \scriptstyle 1 & \scriptstyle 1 & \scriptstyle 0 & \scriptstyle 0 & \scriptstyle 1  \\
			\scriptstyle 0 & \scriptstyle 0 &\scriptstyle 0 & \scriptstyle 1 & \scriptstyle 1 & \scriptstyle 1  \\
			
			\end{pmatrix} $  \\
			
			\hline\hline
		\end{tabular}
	\end{center}
	and
	\begin{center}
		\begin{tabular}{c||r}
			\hline\hline
			$\bullet$ If $(q_{1}, q_{2})\equiv (3, 1) \pmod 4 $ & $\bullet$  If $(q_{1}, q_{2})\equiv (3, 3) \pmod 4$ \\
			\hline
			$ R_{K/k} = \begin{pmatrix}
			\scriptstyle 0 & \scriptstyle 0 & \scriptstyle 0  & \scriptstyle 0 & \scriptstyle 0 & \scriptstyle 1 \\
			\scriptstyle 0 & \scriptstyle 1 & \scriptstyle 0 & \scriptstyle 1 & \scriptstyle 0 &\scriptstyle 0  \\
			\scriptstyle 1 & \scriptstyle 0 & \scriptstyle 1 & \scriptstyle 0 & \scriptstyle 1 & \scriptstyle 1  \\
			\scriptstyle 1 & \scriptstyle 1 & \scriptstyle 1 & \scriptstyle 0 & \scriptstyle 0 & \scriptstyle 1  \\
			\scriptstyle 0 & \scriptstyle 0 &\scriptstyle 0 & \scriptstyle 1 & \scriptstyle 1 & \scriptstyle 1  \\
			
			\end{pmatrix} $
			& $R_{K/k} = \begin{pmatrix}
			\scriptstyle 0 & \scriptstyle 0 & \scriptstyle 0 & \scriptstyle 0 & \scriptstyle 0 & \scriptstyle 1 \\
			\scriptstyle 0 & \scriptstyle 0 & \scriptstyle 0 & \scriptstyle 0 & \scriptstyle 0 & \scriptstyle 1 \\
			\scriptstyle 1 & \scriptstyle 1 & \scriptstyle 1 & \scriptstyle 1 & \scriptstyle 1 & \scriptstyle 0 \\
			\scriptstyle 1 & \scriptstyle 1 & \scriptstyle 1 & \scriptstyle 0 & \scriptstyle 0 & \scriptstyle 1  \\
			\scriptstyle 0 & \scriptstyle 0 &\scriptstyle 0 & \scriptstyle 1 & \scriptstyle 1 & \scriptstyle 1
			\end{pmatrix}. $  \\
			
			\hline\hline
		\end{tabular}
	\end{center}

  So 	$$ \mathrm{rank}(R_{K/k})=\begin{cases} 3, & \text{if $(q_{1}, q_{2})\equiv (3, 3) \pmod 4 $, } \\
	4, & \text{otherwise.}
  \end{cases}$$
	Then
	$$ \hspace{3em} r_{4}(K)=\begin{cases} 1, & \text{if $(q_{1}, q_{2})\equiv (3, 3) \pmod 4 $, } \\
	0, & \text{otherwise.}
	\end{cases}$$	
	\end{proof}
	\begin{example}
	We finish this work by illustrating  the last Theorem with some examples:
		\begin{itemize}
			\item [$ (i) $] For $p =5$, $q_{1} =3$ and $q_{2} =7$, we have $p=5\equiv5\pmod 8$ and $ \left(\frac{5}{3}\right)=\left(\frac{5}{7}\right)=-1$.  From Corollary~((\ref{SHC})-(b)), we obtain $r_{2}(K) =2$. As $(q_{1}, q_{2})=(3, 7)\equiv (3, 3) \pmod 4 $, then, $r_{4}(K)=1$. In fact, $ Cl_2(K)\cong \mathbb{Z}/4\mathbb{Z}\times \mathbb{Z}/2\mathbb{Z} $.
			\item [$(ii)$] Likewise if $ p =13 $, $ q_{1} =19 $ and $ q_{2} =83 $, we obtain $r_{2}(K) =2$ and $r_{4}(K)=1$, but this time, $ Cl_2(K)\cong \mathbb{Z}/16\mathbb{Z}\times \mathbb{Z}/2\mathbb{Z} $.

\noindent  For the case $(q_1, q_2)\not\equiv (3, 3) \pmod 4 $, we have the following examples:
			\item [$ (iii) $] Take $p=37$, $q_{1}=17$ and $q_{2}=29$, we have $p=37\equiv5\pmod 8$ and $ \left(\frac{37}{17}\right)=\left(\frac{37}{29}\right)=-1$.
 From Corollary~((\ref{SHC})-(b)), we obtain $r_{2}(K) =2$. Also $(q_{1}, q_{2})=(17, 29)\equiv (1, 1) \pmod 4 $, so the condition of Theorem~(\ref{Thc}) are satisfied. Hence,  $r_{4}(K)=0$ and $ Cl_2(K)\cong \mathbb{Z}/2\mathbb{Z}\times \mathbb{Z}/2\mathbb{Z}$.
			\item [$(iv)$] For $ p =37 $, $ q_{1} =17 $, $ q_{2} =31 $, we have $p=37\equiv5\pmod 8$, $ \left(\frac{37}{17}\right)=\left(\frac{37}{31}\right)=-1$ and $(q_{1}, q_{2})=(17, 31)\equiv (1, 3) \pmod 4$. we also find $ Cl_2(K)\cong \mathbb{Z}/2\mathbb{Z}\times \mathbb{Z}/2\mathbb{Z}$.
		\end{itemize}
	\end{example}	

\section*{Acknowledgment}
We would like to thank the unknown referee  for his/her several helpful suggestions and for calling our attention
to the missing details.

\end{document}